\documentclass{article}
\usepackage[utf8]{inputenc}
\usepackage{authblk}
\usepackage{indentfirst}
\usepackage{setspace}
\usepackage{enumerate} 
\usepackage{amsmath,amssymb,amsthm,enumerate,graphicx}
\usepackage[all]{xy}
\usepackage{indentfirst}
\usepackage{graphics}
\usepackage{pict2e}
\usepackage{epic}
\usepackage{natbib}
\usepackage{amssymb,latexsym,amsmath}
\usepackage{amsthm}
\usepackage{dsfont}
\usepackage{amsfonts}
\usepackage{mathrsfs}

\usepackage[english]{babel}
\usepackage{array}

\numberwithin{equation}{section}
\usepackage{epstopdf} 
\usepackage{tikz-cd}
\usepackage{tikz}
\usepackage{calligra,mathrsfs}
\usepackage[hidelinks]{hyperref}
\usepackage{multirow}
\usepackage{lipsum}
\usepackage{array}
\usepackage[mathscr]{eucal}
\usepackage{titlesec}
\usetikzlibrary{shapes,arrows}
\usetikzlibrary{decorations.markings}
\usetikzlibrary {patterns,patterns.meta}
\usepackage{mathtools}
\usepackage{cases, caption}
\usepackage{mathalfa}
\usepackage{amscd}
\usepackage[verbose]{placeins}
\usepackage{graphicx}
\graphicspath{ {/admin/desktop/} }
\usetikzlibrary{decorations.pathreplacing,decorations.markings}
\usetikzlibrary{shapes,shapes.arrows,positioning}
\allowdisplaybreaks
\tikzset{->-/.style={decoration={markings, mark=at position #1 with {\arrow{>}}},postaction={decorate}}}

\newcommand{\GL}{\text{GL}}

\newcommand{\R}{\mathbb{R}}

\newcommand{\mg}{M_{\Gamma}}

\newcommand{\Z}{\mathbb{Z}}

\newcommand{\tsh}{\mathfrak{T}}

\DeclareMathOperator{\diff}{\text{Diff}^{+}}

\DeclareMathOperator{\ts}{Twist}
\DeclareMathOperator{\Out}{Out}
\DeclareMathOperator{\map}{Map}
\DeclareMathOperator{\pmap}{PMap}
\DeclareMathOperator{\supp}{supp}
\DeclareMathOperator{\Hom}{Hom}
\DeclareMathOperator{\Fr}{Fr}
\DeclareMathOperator{\Homeo}{Homeo}
\newcommand{\GLR}{\text{GL}^{+}(3,\R)}
\usetikzlibrary {patterns,patterns.meta}
\newtheorem{colo}{Corollary}[section]

\newtheorem{teorema}{Theorem}[section]
\newtheorem{lema}{Lemma}[section]
\newtheorem{prop}{Proposition}[section]

\theoremstyle{remark}
\newtheorem{remark}{Remark}[section]

\title{An explicit section of the Laudenbach-type exact sequence of the big mapping class group of $M_{\Gamma}$}
\author{Jorge Andres Robinson Arrieta\\
  \small Keiser University Flagship Residential Campus\\
  \small West Palm Beach, Florida, United States\\
  \small\texttt{jorge.robinson@keiseruniversity.edu}}
\date{October 2026}

\begin{document}

\maketitle

\begin{abstract}\vspace{0.45cm}
Given a locally finite graph $\Gamma$, Udall proved that the mapping class group of the doubled handlebody  associated to $\Gamma$ is an extension of the mapping class group of $\Gamma $ by a  possibly infinite direct product of $\Z_2$, and  that this extension splits. We give a formula for Udall’s section restricted to the pure mapping class group and, when $\Gamma$ has finitely many ends, on the whole mapping class group. 

\end{abstract}

\section{Introduction}\label{sec:introduction.}

Given a locally finite graph $\Gamma$, $\map(\Gamma)$  denotes the group of proper homotopy equivalences of $\Gamma$ up to proper homotopy. If $\Gamma$ has marked points, we require these homotopy equivalences to fix the marked points of $\Gamma$. Given a smooth oriented manifold $M$, let $\map(M)$ be the group of orientation-preserving diffeomorphisms of $M$ up to isotopy. If $M$ has boundary, we require these diffeomorphisms to fix the boundary pointwise.  Recall
that $\pi_1(SO(3), \mathrm{Id}) \cong \Z_2$ is generated by a loop $\ell:[0, 1] \xrightarrow{} SO(3)$ that rotates $\R^3$ about an axis by one full turn. Let $S^2$ be the standard sphere in $\R^3$, and define $\tau_0:S^2\times [0,1]\xrightarrow{}S^2\times [0,1]$ by $\tau_0(s, t) = (\ell(t) \cdot s, t)$. Let $S\subset M$ be a smoothly embedded 2-sphere. Fix a tubular neighborhood $U\cong S \times [0, 1]$ of $S$. The sphere twist about $S$, denoted $T_S$, is the isotopy class of the diffeomorphism $\tau : M \xrightarrow{}M$ defined by:

\[  \tau(p)=\left\{
\begin{array}{ll}
       \phi\circ \tau_0 \circ \phi^{-1}(p) & \text{ if } p\in U  \\
     p &  \text{ otherwise}, \\
      
\end{array}
\right. \] where $\phi$ is a diffeomorphism from $S^2\times [0,1]$ to $U$.

If $M$ is a compact manifold with a possibly nonempty boundary, $\ts(M)$ is defined to be the subgroup of $\map(M)$ generated by sphere twists.

 We consider the doubled handlebody $M_{\Gamma}$ associated with $\Gamma$.  This is the manifold constructed by gluing two copies of a regular neighborhood of the image of a proper embedding of $\Gamma$ in $\R^3$ along their common surface boundary by the identity map. In the case where $\Gamma$ is compact, we have that $M_{\Gamma}=M_n$ is the connected sum of $n$ copies of $S^2 \times S^1$ for some $n$, and $\map(\Gamma)=\text{Out}(F_n)$. A sphere system is a smooth proper embedding of (possibly infinitely many) disjoint essential non-peripheral 2-spheres in $M_{\Gamma}$. A sphere twist about a sphere system is the simultaneous twist along all its spheres. $\ts(\mg)$ is defined to be the subgroup of $\map(\mg)$ generated by sphere twists on sphere systems of $\mg$. When $\Gamma$ is compact, this definition agrees with the one given above, since sphere systems are then finite.
 
In \cite{1,2}, Laudenbach proved that in this case there exists a short exact sequence:
$$1 \xrightarrow{} \ts(M_n)\xrightarrow{} \map(M_n)
\xrightarrow{\rho} \Out(F_n) \xrightarrow{}1,$$ 
 where $\ts(M_n) \cong \Z_2^{n}$  is generated by the sphere twists about
the core spheres $S^{2} \times *$. Brendle--Broaddus--Putman proved
in \cite{3}  that this short exact sequence splits by constructing a crossed homomorphism $\tsh:\map(M_n)\xrightarrow{}\text{Twist}(M_n).$ This crossed homomorphism determines a section $s:\Out(F_n)\xrightarrow{}\map(M_n)$ that is given by $s([\phi])=\mathfrak{T}([f^{-1}])[f]$, where $f$ is a  diffeomorphism of $M_n$ with $\rho([f])=[\phi]$. The map $s$ is well defined and does not depend on the choice of the lifts of the elements in $\Out(F_n)$.

In $\cite{4}$ we provided an explicit formula for the section $s$ of $\rho$ in terms of certain diffeomorphisms $F_{i,j},G_j$ of $M_{n}$ defined in Section \ref{sec: construction of diffeomorphisms}.

In \cite{7}, Udall proved that in the general case there is an exact sequence: 

$$1 \xrightarrow{} \ts(M_{\Gamma})\xrightarrow{} \map(M_{\Gamma})
\xrightarrow{\Psi} \map(\Gamma) \xrightarrow{}1.$$ 

He also proved  that the short exact sequence  splits topologically, giving an isomorphism:
$$ \map(\mg)\cong \ts(\mg)\rtimes\map(\Gamma),$$
as topological groups. 
Udall proved this by noticing that the crossed homomorphism constructed in \cite{3}  can be extended to a crossed homomorphism: $$\mathfrak{T}: \map(\mg)\xrightarrow{} \ts(\mg).$$ 

Since $\tsh$ restricts to the identity on $\ts(\mg)$, it determines a continuous section 
$s:\map(\Gamma)\xrightarrow{}\map(\mg)$ that is given by $s([\phi])=\mathfrak{T}([f^{-1}])[f]$, where $f$ is a diffeomorphism of $\mg$ with $\Psi([f])=[\phi]$. The map $s$ is well defined and does not depend on the choice of the lifts of the elements in $\map(\Gamma)$. 

Let $E(\Gamma)$ be the set of ends of $\Gamma$, and let $\pmap(\Gamma)$ be the subgroup of $\map(\Gamma)$ consisting of mapping classes that do not permute the ends. The purpose of this paper is to provide an explicit formula for $s$ on $\pmap(\Gamma)$ and, when $|E(\Gamma)| < \infty$, on all of $\map(\Gamma)$. Let $\map_c(\Gamma)$ be the subgroup of $\map(\Gamma)$ consisting of compactly supported mapping classes. Domat--Hoganson--Kwak showed in \cite{Domatgenerating} that  $\pmap(\Gamma)$ is topologically generated by the closure of $\map_c(\Gamma)$ and the loop shifts. When $|E(\Gamma)| < \infty$, $\map(\Gamma)$ is generated by $\pmap(\Gamma)$ together with the end swaps $p_{i,j}$.
Since $s$ is a continuous homomorphism,  it is enough to compute it on compactly supported classes, on loop shifts, and on end swaps, defined below. Elements of the closure of $\map_c(\Gamma)$ are then handled by taking limits. Every compactly supported class is supported on a finite subgraph of $\Gamma$ that is homotopy equivalent, relative to its points of intersection with the rest of $\Gamma$, to a graph $G_{n,k}$, which we now define.

The standard thorned rose $G_{n,k}$ is a bouquet of $n$ circles with vertex $x_0$ together with $k$ free edges attached to $x_0$ and these free edges have marked endpoints labeled as $\omega_1, \ldots, \omega_{k}$, as shown in the picture below.
\begin{figure}[ht]
    \centering
    \includegraphics[width=0.4\textwidth]{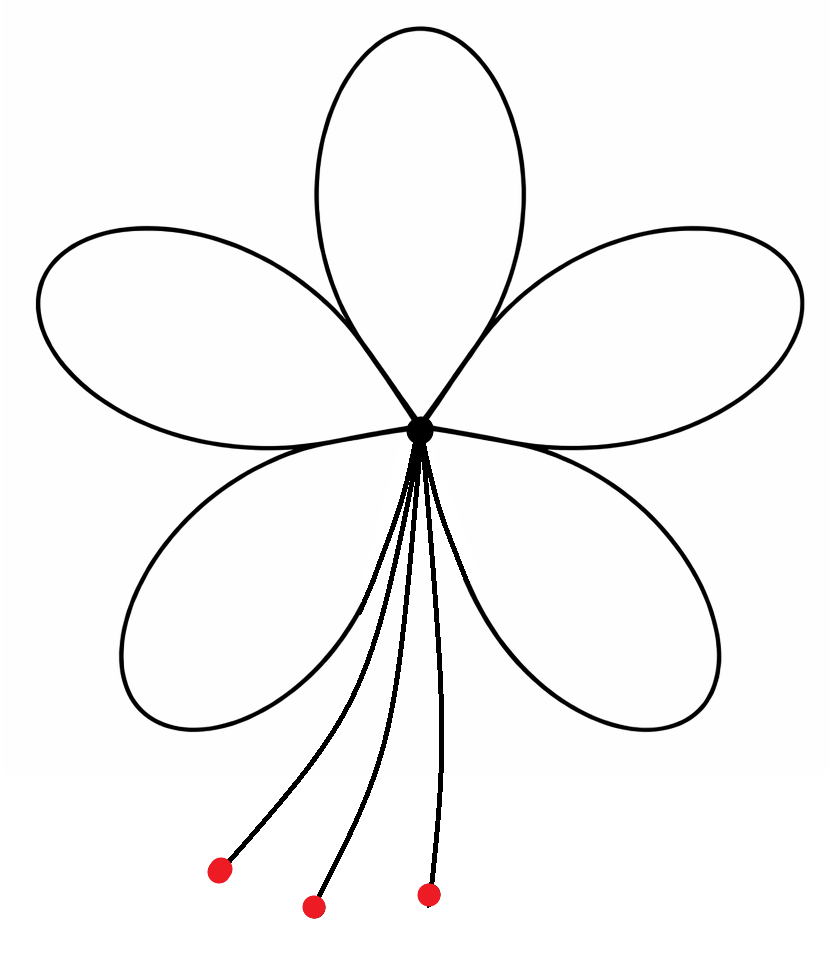}
    \caption{Topological thorned rose $G_{5,3}$.}
\end{figure}
\FloatBarrier
 Let $M_{n,k}$ be the doubled handlebody of $G_{n,k}$, with one boundary sphere per marked point. Hatcher--Vogtmann introduced the following short exact sequence in \cite{5}:
$$ 1\xrightarrow{} \ts(M_{n,k})\xrightarrow{}\map(M_{n,k})\xrightarrow{\rho_{n,k}}\map({{G}_{n,k}})\xrightarrow{}1.$$

To compute $s$ on the closure of $\map_c(\Gamma)$, we proceed as follows. First, we prove that the Hatcher--Vogtmann exact sequence splits, and provide a formula for sections $s_{n,k}$ of $\rho_{n,k}$ in terms of certain diffeomorphisms $F_{i,j},G_j,\Phi_{(a_i,J_l)}$ of $M_{n,k}$ defined in Section \ref{sec: construction of diffeomorphisms}.

\begin{teorema}\label{Hatcher Vogtmann split}
For $k\ge 1$, the Hatcher--Vogtmann short exact sequence is a split short exact sequence with $\ts(M_{n,k})\cong \Z_2^{n+k-1}.$ Moreover, the map $$s_{n,k}:\map(G_{n,k})\xrightarrow{} \map(M_{n,k}),$$ that on Nielsen generators $[R_{i,j}],[I_j]$, and on word maps $[\varphi_{(a_i,J_l)}]$ defined later, for $1\le i, j\le n, i\neq j,$ and $1\le l\le k $, is given by:
    $$ s_{n,k}([R_{i,j}])=[F_{i,j}] ,  s_{n,k}([I_j])=[G_j],  \text{ and }  s_{n,k}([\varphi_{(a_i,J_l)}])=[\Phi_{(a_i,J_l)} ],$$ is a section of $\rho_{n,k}$, where  $F_{i,j}$ , $G_j$  and $\Phi_{(a_i,J_l)} $ are certain diffeomorphisms of $M_{n,k}$ defined in Section \ref{sec: construction of diffeomorphisms}. 

\end{teorema}

When $k=1$, the splitting of the Hatcher--Vogtmann exact sequence was already proved in \cite{3}.  It follows from the work of Hatcher--Wahl in \cite{18} that $\ts(M_{n,k})\cong \Z_2^{n+k-1},$ but in this work we establish this in a more direct way. The proof of Theorem \ref{Hatcher Vogtmann split} requires the introduction of a relative crossed homomorphism for manifolds with boundary.

 Next, we assemble the sections $s_{n,k}$ of the exact Hatcher--Vogtmann sequence to get a formula for the section of $\Psi$ as a limit in the appropriate sense of the sections $s_{n,k}.$ Let 
 $\theta:\map(M_{n,k})\xrightarrow{}\map_c(M_\Gamma)$ be the map that extends a diffeomorphism on $M_{n,k}$ by the identity as defined in Section $\ref{sec: section s and snk}.$ Then we prove the following theorem:

\begin{teorema}\label{thm limit si}
   Let $\Gamma$ be a locally finite graph, and let $[\phi]\in\overline{\map_c(\Gamma)}$.  Let $\{[\phi_i]\}_{i\in\mathbb{N}}$ be a sequence in $\map_c(\Gamma)$ converging to $[\phi],$ and let $\alpha_i: \Gamma_i\xrightarrow{}G_{n_i,k_i}$ be homotopy equivalences  where $\Gamma_i$ is a thorned subgraph with  $\text{supp}(\phi_i)\subseteq \Gamma_i.$   Let $s_i^{\prime}:=\theta \circ s_{n_i,k_i}\circ (\alpha_{i})_{*}$, and let $s_i([\phi_i]):=s_i^{\prime}([\phi_{i}\Bigr|_{\Gamma_i}])$. Then:
   
   
    $$s([\phi])=s([\lim\limits_{i \to \infty}\phi_i])=\lim\limits_{i \to \infty} s_i([\phi_{i}]).$$
\end{teorema}

Since the mapping classes $s_i([\phi_{i}])$ are explicitly described by Theorem~\ref{Hatcher Vogtmann split}, Theorem~\ref{thm limit si} allows us to express $s([\phi])$ as a limit of these maps.

Denote by $E_{\ell}(\Gamma)$ the set of ends that are accumulated by loops.
Suppose that $|E_{\ell}(\Gamma)|\ge2,$ and let $e_1,e_2\in E_{\ell}(\Gamma)$. Give $\R$ the standard CW structure with vertex set $\Z$. Let $e:\R\hookrightarrow \Gamma$ be a graph embedding such that $e[0,\infty)$ is a ray converging to $e_1$ and $e(-\infty,0]$ is a ray converging to $e_2$. Assume that there is at least one loop $\ell_i$ attached to each vertex $e(i)$ of $e(\R)$, and there is a set (possibly empty) of  edges that are attached to each vertex $e(i)$ as in the picture below:

\begin{figure}[ht]
   \centering
     \includegraphics[width=0.9\textwidth]{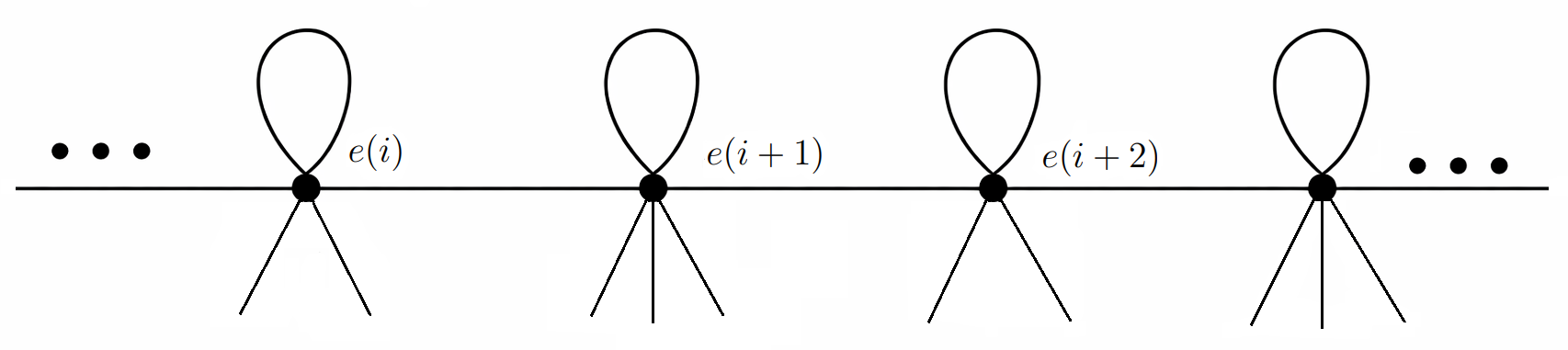}
    \caption{Image $e(\R)$ of a graph embedding $e:\R\hookrightarrow \Gamma$ with at least one loop $\ell_i$ attached at each vertex.}
\end{figure}
 \FloatBarrier

The loop shifts associated  to $e_1$ and $e_2$ are the elements of the group generated by the graph map $h$ that sends the vertex $e(i)$ to $e(i+1)$, the loop $\ell_{i}$ to $\ell_{i+1}$,  an edge $\tilde{e}$ attached to $e(i)$ to the path $\tilde{e}\cup [e(i),e(i+1)]$, and is the identity elsewhere  as in the picture below: 

\begin{figure}[ht]
   \centering
    \includegraphics[width=0.9\textwidth]{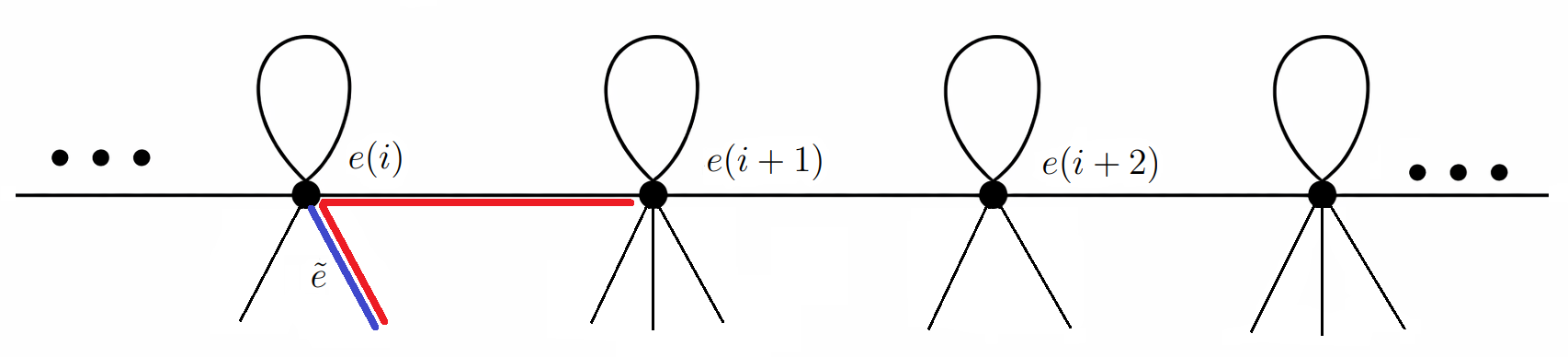}
    \caption{Action of $h$ on $e(\R)$.}
\end{figure}
 \FloatBarrier

As stated in the following theorem, we found lifts for the generating loop shifts defined above:

\begin{teorema}\label{thm section lift of h}
Let $|E_{\ell}(\Gamma)|\ge 2$ and let $e_1,e_2\in E_{\ell}(\Gamma)$ be ends as described above. Let $h$ be a generating loop shift map associated to $e_1$ and $e_2$. Then:
$$s([h])=[H],$$ where $H$ is a certain diffeomorphism  of $\mg$ defined in Section \ref{sec: construction of diffeomorphisms}.

\end{teorema}

 Since $\pmap(\Gamma)$ is topologically generated by the closure of the compactly supported mapping classes and  by the loop shifts,  Theorems \ref{thm limit si} and  \ref{thm section lift of h} describe the section $s$ of $\Psi$ restricted to $\pmap(\Gamma)$.


If $|E(\Gamma)|<\infty$ and $|E_{\ell}(\Gamma)|= 0,$ then the canonical form of $\Gamma$ (see Section \ref{sec:generating sets}) can be further homotoped such that $\Gamma$ looks as follows:

\begin{figure}[ht]
    \centering
    \includegraphics[width=0.7\textwidth]{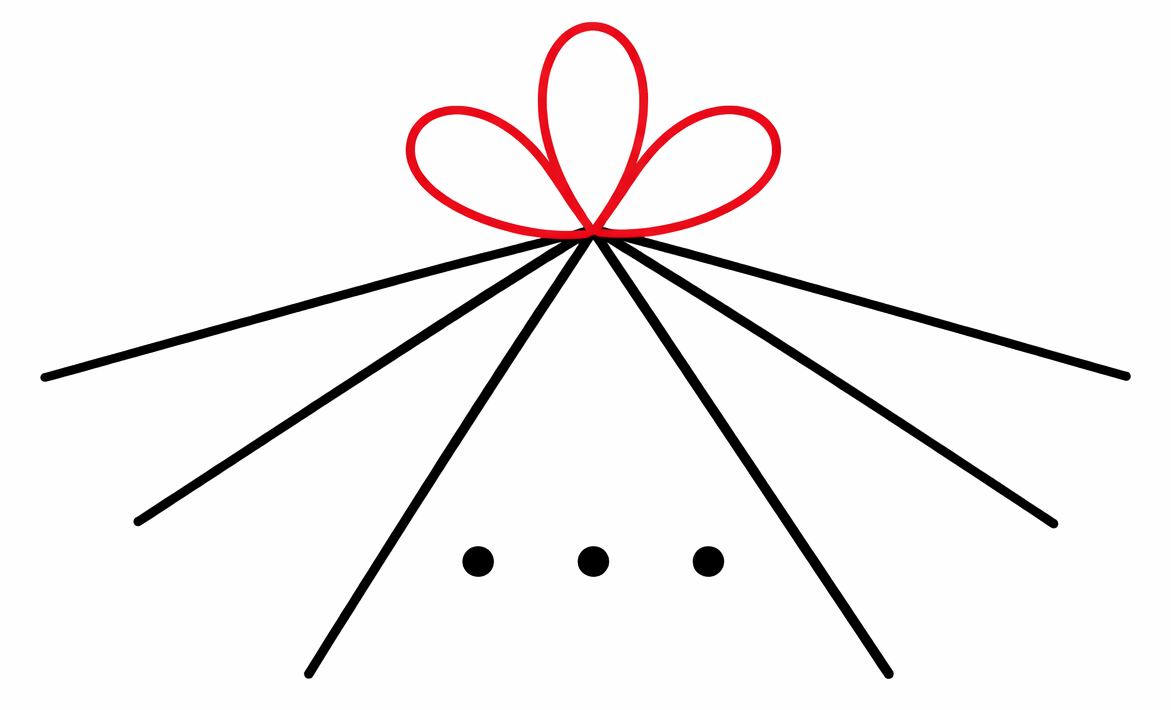}
    \caption{Canonical form of $\Gamma$ when $|E_{\ell}(\Gamma)|= 0$.}
\end{figure}
 \FloatBarrier
If $|E(\Gamma)|<\infty$ and $|E_{\ell}(\Gamma)|\ge 1,$ then the canonical form of $\Gamma$ can be further homotoped such that $\Gamma$ looks as follows:

\begin{figure}[ht]
    \centering
  \includegraphics[width=0.7\textwidth]{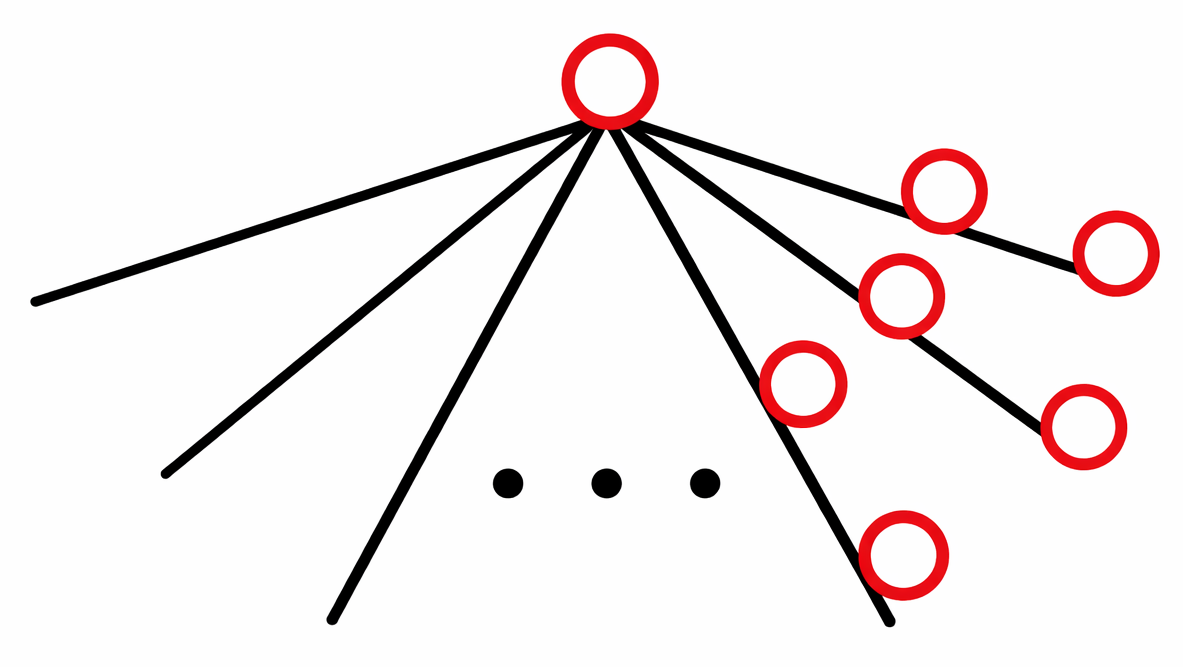}
    \caption{Canonical form of $\Gamma$ when $1\le |E_{\ell}(\Gamma)|< \infty$.}
\end{figure}
\FloatBarrier

We say that two ends have the same type if both ends are accumulated by loops or both are not. 
By $p_{i,j}$ we denote the homotopy equivalence of $\Gamma$ that  induces a swapping of the $i$-end and the $j$-end of the same type.  Let $\Gamma$ be in canonical form, and let $e_i$ and $e_j$ be two ends of the
same type, with vertices $v_1, v_2, \ldots$ and $v'_1, v'_2, \ldots$,
respectively. Then $p_{i,j}$ is the graph map that swaps $v_k$ and $v'_k$ for
every $k$ and is the identity outside $e_i \cup e_j$. If $e_i$ and $e_j$ are
accumulated by loops, with loops $\ell_k$ and $\ell'_k$ attached at $v_k$ and
$v'_k$, then $p_{i,j}$ also swaps $\ell_k$ and $\ell'_k$. We show this map in
the picture below.


\begin{figure}[ht]
    \centering
  \includegraphics[width=0.5\textwidth]{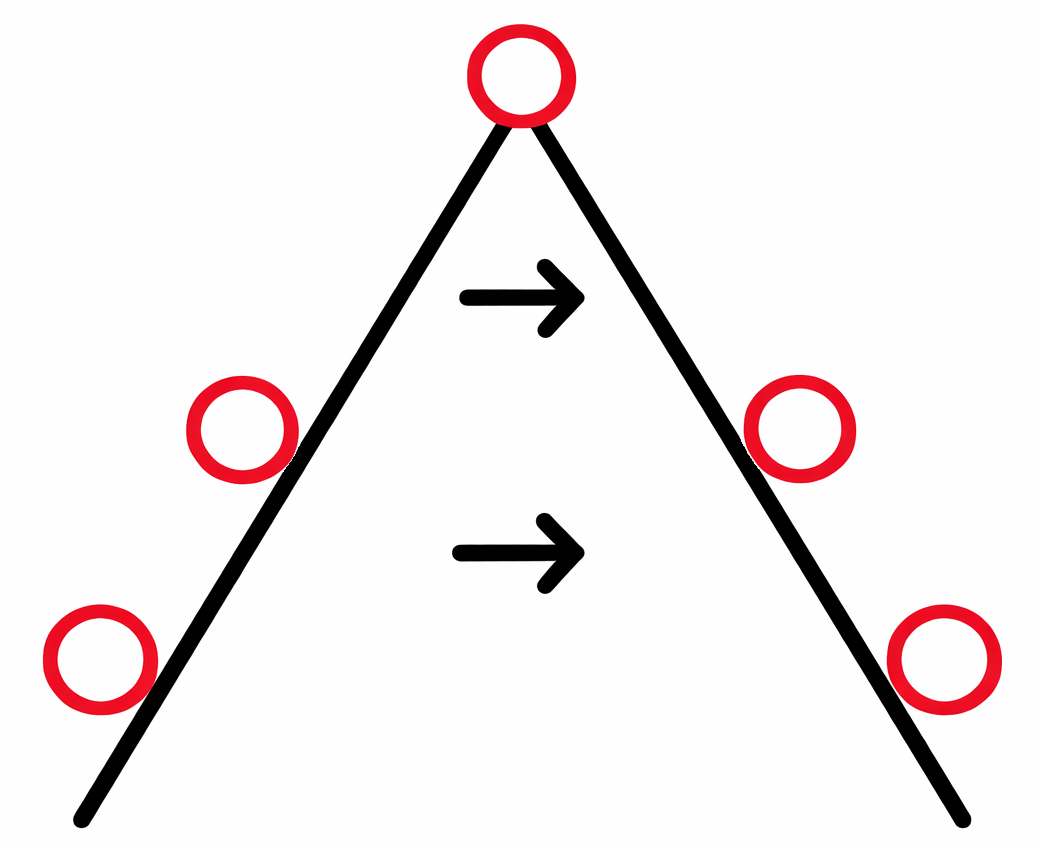}
    \caption{Action of $p_{i,j}$ on  the canonical form of $\Gamma$ when $1\le |E_{\ell}(\Gamma)|< \infty$.}
\end{figure}
\FloatBarrier

We found lifts of $p_{i,j}$ as stated in the following theorem:

\begin{teorema}\label{thm section lift of pij}
Suppose that $|E(\Gamma)|<\infty$. Let $ e_i,e_j\in E_{\ell}(\Gamma)$ or $e_i,e_j\in E(\Gamma)\setminus{}E_{\ell}(\Gamma)$.
   Then:

$$ s([p_{i,j}])=[P_{i,j}],$$  where $P_{i,j}$ is a certain diffeomorphism of $\mg$ defined in Section \ref{sec: construction of diffeomorphisms}.

\end{teorema}

When $|E(\Gamma)|<\infty$, we have that $\map(\Gamma)$ is generated by all the pure mapping classes and by all $p_{i,j}$. Thus, in this case, Theorems \ref{thm limit si}--\ref{thm section lift of pij} describe the section $s$ of $\Psi$.   
$$ $$

\textbf{Outline}: The paper is organized as follows. In Section \ref{sec:generating sets} we recall generating sets for $\mathrm{PMap}(\Gamma)$ and, when $E(\Gamma)$ is finite, for $\mathrm{Map}(\Gamma)$, as well as for $\mathrm{Map}(G_{n,k})$ in terms of Nielsen generators and word maps. In Section \ref{sec: construction of diffeomorphisms} we construct the diffeomorphisms $F_{i,j}$, $G_j$, $\Phi_{(a_i,J_l)}$, $H$ and $P_{i,j}$ that lift these generators. In Section \ref{sec:Twisting crossed} we define the relative twisting
crossed homomorphism $\tsh_{n,k}$ for $M_{n,k}$, identify
$\mathrm{Twist}(M_{n,k})$ with
$H^1(M_{n,k},\partial M_{n,k})\cong\mathbb{Z}_2^{\,n+k-1}$,
and deduce that the Hatcher--Vogtmann sequence splits. In Section \ref{sec:Computation twisting crossed} we show that the twisting crossed homomorphism vanishes on each of these
lifts: $\tsh_{n,k}$ on $F_{i,j}$, $G_j$ and $\Phi_{(a_i,J_l)}$, and $\tsh$ on $H$ and $P_{i,j}$. Finally, in Section \ref{sec: section s and snk} we combine these computations to prove Theorem \ref{Hatcher Vogtmann split}, express the section on the closure of $\mathrm{Map}_c(\Gamma)$ as a limit of the sections $s_{n,k}$ (Theorem \ref{thm limit si}), and prove Theorems \ref{thm section lift of h} and \ref{thm section lift of pij}. Throughout, all (co)homology groups have $\Z_2$ coefficients.


\section{Generating set for $\map(\Gamma)$ and $\map(G_{n,k})$}\label{sec:generating sets}

Given a locally finite graph $\Gamma$, denote by $E(\Gamma)$ the set of ends of $\Gamma$ and denote by $E_{\ell}(\Gamma)$ the set of ends accumulated by loops of $\Gamma$.
Denote by $\Homeo(E(\Gamma),E_{\ell}(\Gamma))$ the group of homeomorphisms of $E(\Gamma)$ that preserve $E_{\ell}(\Gamma)$.  $\pmap(\Gamma)$ is defined to be the kernel of the action of $\map(\Gamma)$ on   $E(\Gamma)$. The following short exact sequence was studied in \cite{Domatgenerating}:
$$1 \xrightarrow{}\pmap(\Gamma)\xrightarrow{} \map(\Gamma) \xrightarrow{}\Homeo(E(\Gamma), E_{\ell}(\Gamma)) \xrightarrow{}1.$$

Thus, $\map(\Gamma)$ is generated by a set of generators of $\pmap(\Gamma)$ and by  chosen lifts of a set of generators of $\Homeo(E(\Gamma),E_{\ell}(\Gamma))$. By $\map_c(\Gamma)$ we denote the mapping classes that are compactly supported. 

Let $\Gamma_c$ denote the core graph of $\Gamma,$  and let $T$ denote the maximal tree of $\Gamma$, which is included in  $\Gamma_c$. 
Domat--Hoganson--Kwak showed in \cite{Domatgenerating} that we can select a countable collection $\mathcal{A} = \{A_i\}_{i\in I}$ of cylinder sets of $E_{\ell}(\Gamma)$  and a map  $\iota : \mathcal{A}\xrightarrow{} T$, such that the closures in $\Gamma_c$ of the complementary components of the image of $\mathcal{A}$ are one-ended. Additionally, they showed that $\Gamma$ can be properly homotoped so that these one-ended components are replaced by certain “grids of loops”, where it is easier to identify commuting loop shifts. Moreover, they showed that there is a map $\Pi:\pmap(\Gamma)\xrightarrow{}\prod\limits_{\alpha} \Z$, for some $\alpha$, defined using flux maps, which were first defined for locally finite, infinite graphs by Domat--Hoganson--Kwak in \cite{Domatcoarse}. 

\begin{teorema}{\cite[Theorem B]{Domatgenerating}}
\label{thm basis for pmap}
 Let $\Gamma$ be a locally finite graph. Let $\alpha = max\{0, |E_{\ell}(\Gamma)|-1\}$  for $|E_{\ell}(\Gamma)| < \infty$
and $\alpha = \aleph_0$ otherwise.
 Then we have the following short exact sequence:
$$1\xrightarrow{}\overline{ \map_c(\Gamma)}\xrightarrow{} \pmap(\Gamma) \xrightarrow{\prod} \prod\limits_{\alpha} \Z\xrightarrow{} 1.$$
Moreover, this short exact sequence splits. In particular, we have: $$\pmap(\Gamma) = \overline{\map_c
(\Gamma)} \rtimes \prod\limits_{\alpha} \Z.$$ 
\end{teorema}

Here, $\prod\limits_{\alpha} \Z$ is topologically generated by the image under $\Pi$ of some commuting loop shifts $h_{\alpha_i}$, $1\le i \le \alpha$. We will see that our calculations hold not only for these commuting loop shifts but for all loop shifts.


The standard form of a locally finite graph $\Gamma$, as defined by Domat--Hoganson--Kwak in \cite{Domatgenerating}, is a tree with loops attached at some of the vertices. In certain cases, it is more convenient to represent a graph in a different form.
We consider the canonical form of a graph $\Gamma$ introduced by Ayala--Dominguez--Marquez--Quintero in \cite{ayala}. We say that a graph $\Gamma$ is in canonical form if $\Gamma$ can be described as follows:
\begin{itemize}

\item The maximal tree of $\Gamma$ can be identified with a subset of the standard Cantor tree with root vertex $v_0$.

\item  There is a finite bouquet of circles attached at the root vertex $v_0$.

\item Each end not accumulated by loops is represented by a proper embedding of $\R_{\ge 0}$ into $\Gamma$ as a CW complex starting at $v_0$.

\item Each end accumulated by loops is represented by a proper embedding of $\R_{\ge 0}$ into $\Gamma$ as a CW complex starting at $v_0$, and there is exactly one loop edge attached at each vertex.  

\end{itemize}

We show an example of a graph in canonical form in the following picture:

\begin{figure}[ht]
    \centering
\includegraphics[width=0.5\textwidth]{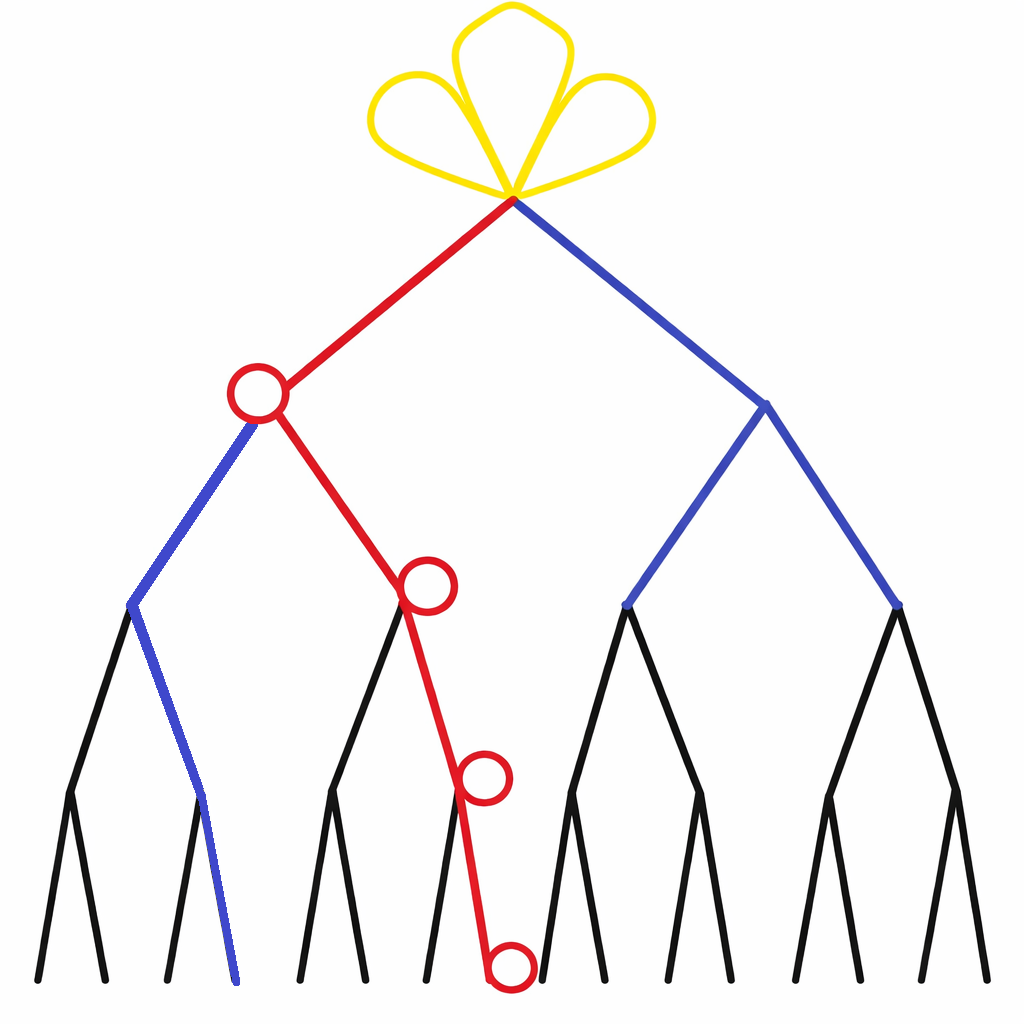}
    \caption{The Cantor tree is pictured in black, and the graph $\Gamma$ consists of the bouquet of circles in yellow, the ends not accumulated by loops in blue, and the ends accumulated by loops in red.}
\end{figure}
\FloatBarrier

 By \cite[Proposition 2.2]{ayala}, any locally finite graph $\Gamma$ is proper homotopy equivalent to a graph in canonical form. Thus, for a graph with finitely many ends in canonical form we have that $\Homeo(E(\Gamma),E_{\ell}(\Gamma))$ is generated by all the end swaps $p_{i,j}$ of ends of the same type. Gathering these propositions we obtain:

\begin{prop}
 \label{thm basis map}   
When $|E(\Gamma)|<\infty$, $\map(\Gamma)$ is topologically generated by the closure of $\map_c(\Gamma)$, all the loop shifts $h$, and  all the maps $p_{i,j}$.
\end{prop}

A thorned graph $\Gamma$ is a compact graph of rank $n$  with $k$ free edges attached to the graph and these free edges have marked endpoints labeled as $\omega_1, \ldots, \omega_{k}$, and $\Gamma$ is homotopy equivalent to $G_{n,k}$ relative to the marked points.  Thus,  there  is an isomorphism $\alpha_{\Gamma}:\map(\Gamma)\xrightarrow[]{} \map(G_{n,k})$, where $n=\text{rk}(\pi_1(\Gamma))$ and $k$ is the number of marked points in $\Gamma.$
 If $\Gamma$ is in standard form, an embedded thorned graph $\Gamma^{\prime}$ in $\Gamma$ looks as follows:

\begin{figure}[ht]
    \centering
    \includegraphics[width=0.9\textwidth]{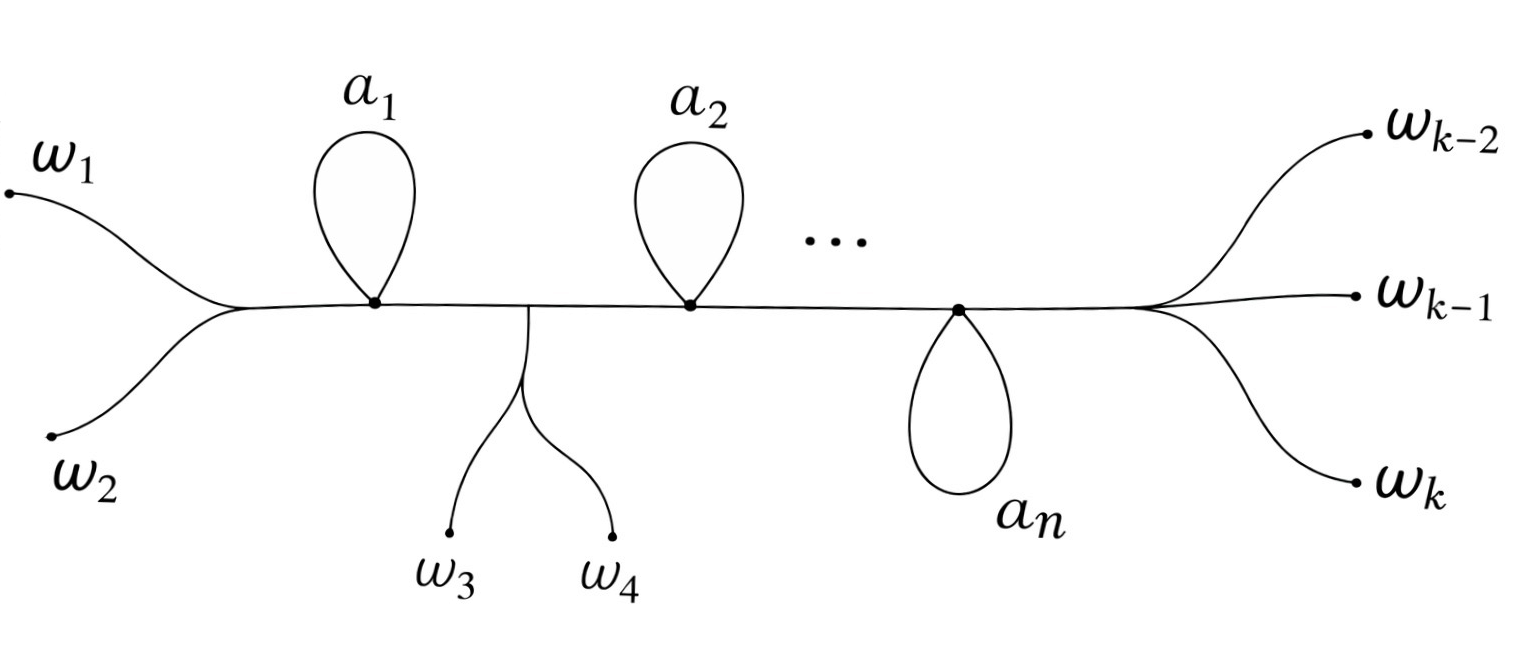}
    \caption{Embedded thorned graph $\Gamma^{\prime}$ in $\Gamma,$ when $\Gamma$ is in standard form.}
\end{figure}
\FloatBarrier

Consider the set of compactly supported homotopy equivalences  $\map_c(\Gamma)$.    
The support of any compactly supported homotopy equivalence is contained in some thorned graph $\Gamma^{\prime}$ embedded in $\Gamma,$  whose marked points are midpoints of edges in $\Gamma.$  A compactly supported homotopy equivalence of $\Gamma$ is then the same as the extension of a homotopy equivalence of an embedded thorned graph $\Gamma^{\prime}$ fixing the marked points, using the identity in the complement of $\Gamma^{\prime}$. 

In \cite{5}, Hatcher--Vogtmann proved that there is a short exact sequence of the form:

$$ 1\xrightarrow{}\pi_1(G_{n,k})\xrightarrow{j}\map(G_{n,k})\xrightarrow{}\map(G_{n,k-1})\xrightarrow{}1,$$ where $j(w)$ is the point push map over $w$, supported on the free edge corresponding to the marked point $\omega_k$. Point push maps in a graph are explicitly defined as word maps in \cite[Definition 2.11]{Domatgenerating}. For convenience, we recall the definition of a word map here.
 Orient and enumerate the loops $\{a_1,...,a_n\}$ in $G_{n,k}$ and recall that $x_0$ is the vertex  of $G_{n,k}$. Enumerate the free edges $\{e_1,...,e_k\}$ in $G_{n,k}$.  Now, identify $\pi_1(G_{n,k}, x_0)$ with $\langle a_1,...,a_n\rangle$.
Let $J_l \subset e_l$ be a connected subset of the $l$-th free edge $e_l \in  G_{n,k}$ and identify $J_l$ with the interval $[0, 1]=[0,1/4]\cup[1/4,3/4]\cup[3/4,1]$. We orient $[0, 1]$ so that 0 is farther from the basepoint. 
Given $w=a^{\epsilon_1}_{i_1}\cdots a^{\epsilon_m}_{i_m}$ we can define a word map, denoted by $\varphi_{(w,J_l)}$, supported on $J_l$ as follows: the interval [0,1/4] is mapped to the  edge $[\omega_l,x_0]$, the interval $[ 1/4,3/4]$ is mapped to the path associated to $w$, and the
interval $[ 3/4, 1]$ is mapped to the edge $[x_0,\omega_l]$. $\varphi_{(w,J_l)}$ is defined to be the identity outside of $J_l$. For example, consider the word map $\varphi_{(a_1,I)}$ as in the picture below:

\begin{figure}[ht]
    \centering
    \includegraphics[width=0.5\textwidth]{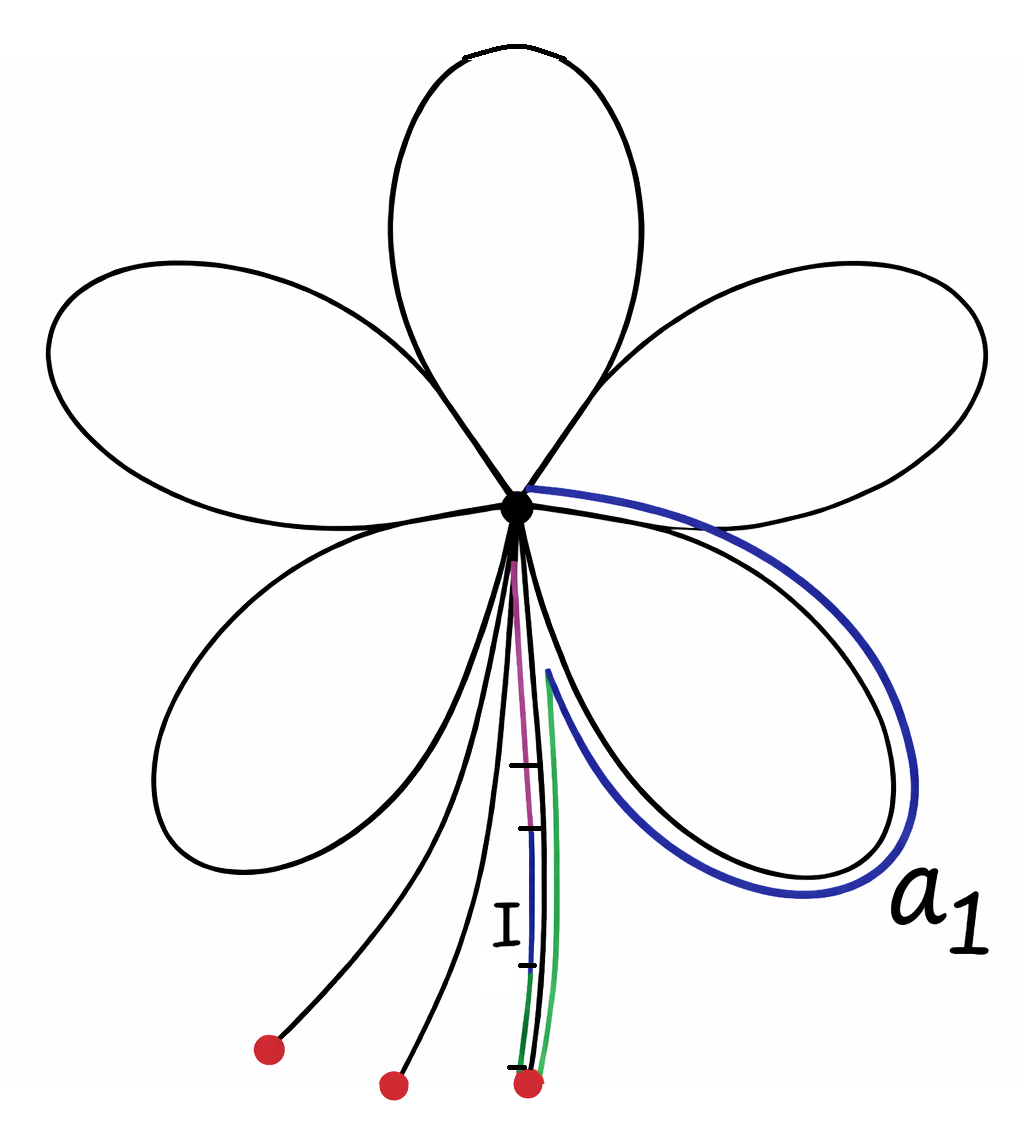}
    \caption{Action of $\varphi_{(a_1, I)}$ on $I$.}
\end{figure}
\FloatBarrier

 It is shown in \cite[Ch. I, Prop. 4.1]{Lyndon} that $\Out(F_n)$ is generated by the classes $[R_{i,j}]$ and $ [I_j]$, for $1\le i,j\le n$ and $i\neq j$, where:

\begin{equation*}
    R_{i,j}(a_k) = \begin{cases}
              a_ka_j & \text{if } k=i\\
               a_k & \text{if } k\neq i\\
               
          \end{cases}, \text{ and }
I_j(a_k) = \begin{cases}
              a_k^{-1} & \text{if } k=j\\
               a_k & \text{if } k\neq j.\\
               
          \end{cases}
\end{equation*}

 Thus, from the Hatcher--Vogtmann short exact sequences, we obtain the following proposition:

\begin{prop}
\label{generating set for map(gamma)}
Let $\Gamma$ be a thorned graph with  $n=\text{rk}(\pi_1(\Gamma))$ and $k$  marked points. Then, $\map(\Gamma)\cong\map(G_{n,k})$ is generated by the equivalence classes of the Nielsen generators $R_{i,j}$, $I_j$,  and by the equivalence classes of the word maps $\varphi_{(a_i,J_l)}$ based on the free edges of $G_{n,k}.$
   
\end{prop}

\section{Construction of the maps $ F_{i,j}, G_j, \Phi_{(a_i,J_l)}, H$, and  $P_{i,j}$}\label{sec: construction of diffeomorphisms}

Let $TM$ be the tangent bundle of $M$  and let $\Fr(TM)$ be the principal $\GLR$-bundle of oriented frames of $TM$. Here, for every $x\in M$, $\Fr(TM)_{x}$ is the space of orientation-preserving isomorphisms $T:\R^{3}\xrightarrow{} T_{x}M$.
 An oriented trivialization of $TM_{\Gamma}$ is a continuous section $\sigma: \mg \xrightarrow{}\text{Fr}(T\mg)$ of the bundle $\text{Fr}(T\mg)\xrightarrow{}\mg$. It is stated in \cite[Theorem 1 of Section VII]{Kirby} that all oriented 3-manifolds have a trivialization. We fix a trivialization $\sigma_0$ of $T\mg$. We always assume that when we choose a diffeomorphism from a neighborhood $N^{\prime}$ in $\R^3$ to a neighborhood $N$ in $\mg$, $N^{\prime}$ has the standard trivialization $\sigma^{\prime}_0$ coming from $\R^3$ and that the diffeomorphism will, up to homotopy,  map $\sigma^{\prime}_0$ to $\sigma_0$. Let $\psi:[0,1]\xrightarrow{} [0,1]$ be a smooth function with $\psi(r)=1$ in [0,1/3], $\supp\psi\subseteq [0,2/3)$ and decreasing.

The smooth manifold $\mg$ can be described as the infinite union of the pieces $M_{n_i,k_i}$, where each piece can be described by removing $2n_i+k_i$ open balls from $S^3$, and then gluing the boundary spheres of the $2n_i$ balls in pairs. A picture of one of these pieces is shown below.

\begin{figure}[ht]
    \centering
   
\includegraphics[width=1\textwidth]{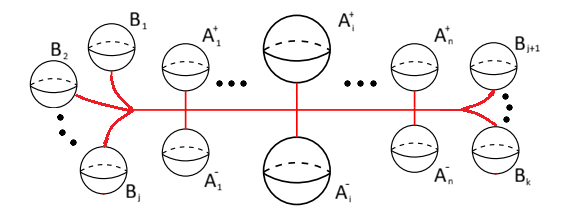}
\caption{Manifold $M_{n,k}$ diffeomorphic to $S^3$ minus $2n+k$ open balls with $2n$ balls glued in pairs.}
\end{figure}
\FloatBarrier



The constructions of $F_{i,j}$ and $ G_j$ were given in \cite{4}, but we recall the definitions here for convenience.  To construct $F_{i,j}$, choose a neighborhood $N_{i,j}$ of $a_j$ diffeomorphic to a solid torus minus an open ball whose boundary is $A_i^{+}$ as in the picture below:

\begin{figure}[ht]
    \centering
  \includegraphics[width=0.65\textwidth]{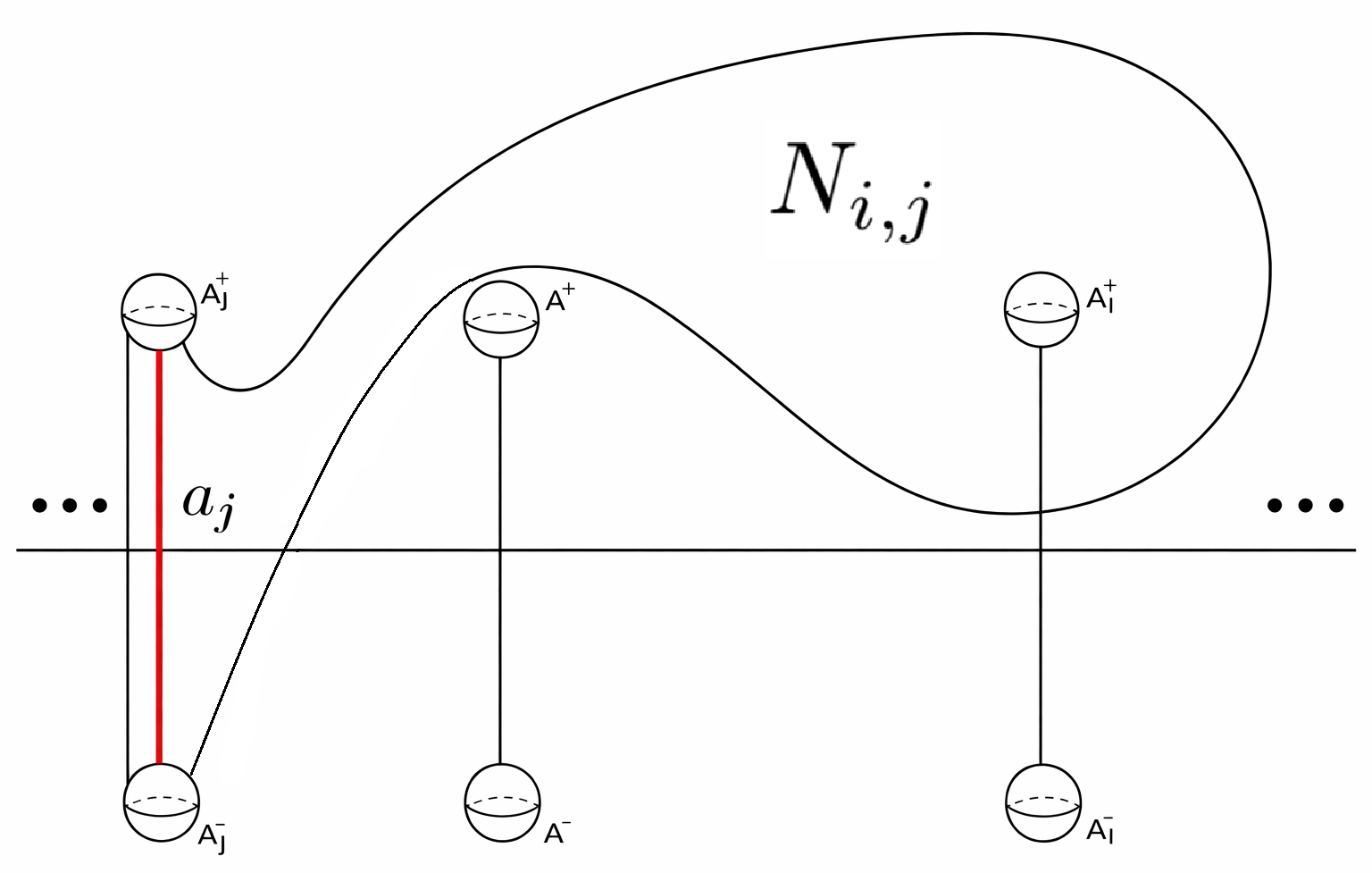}
    \caption{Neighborhood $N_{i,j}$ of $a_j$ in $M_{\Gamma}$.}
    \label{fig:placeholder}
\end{figure}
\FloatBarrier

Choose a diffeomorphism $\phi:N_{i,j}^{\prime}\xrightarrow{}N_{i,j}$ that respects the trivializations, where $N_{i,j}^{\prime}\subseteq\R^3$ is equal to $D^2\times S^1$ minus an open ball with boundary corresponding to  $A_i^{+}$, which is contained in $\{(x,y)|x^2+y^2<1/9\}\times S^1$, where $*\times S^1$ is freely homotopic to $a_j$, as depicted in the picture below:

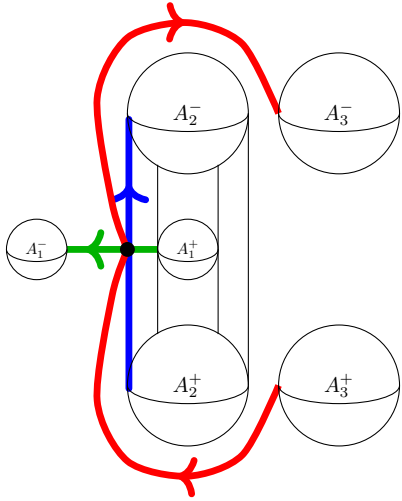
\begin{figure}[h]
\centering

\begin{tikzpicture}[scale=0.4, transform shape]
\draw[line cap=round,->-=0.76, blue, line width=0.9mm] (-1.95,-9.6) --(-1.95,-0.7);
 \draw [->-=0.67, red, line width=0.9mm] plot [smooth ] coordinates {(-2,-5) (-2.5,-3.5) (-3,0) (-2,2) (0,2.5) (1.8,1.8) (3,-0.5)};
 \draw [->-=0.35, red, line width=0.9mm] plot [smooth ] coordinates {(3,-9.5)(1.8,-11.8) (0,-12.5)  (-2,-12) (-3,-10) (-2.5,-6.5)  (-2,-5) };
\draw (2,-0.5)--(2,-9.5);
\draw[->-=0.8, black!30!green, line width=0.9mm] (-1,-5)--(-4,-5);
  \draw (0,-0.5) circle (2cm)node[ scale=2]{$A_2^{-}$};
  \draw (-2,-0.5) arc (180:360:2 and 0.6);
  \draw (0,-9.5) circle (2cm)node[ scale=2]{$A_2^{+}$};
  \draw (-2,-9.5) arc (180:360:2 and 0.6);
  \draw (0,-5) circle (1cm)node[ scale=1.5]{$A_1^{+}$};
  \draw (-1,-5) arc (180:360:1 and 0.4);
  \draw (-5,-5) circle (1cm)node[ scale=1.5]{$A_1^{-}$};
  \draw (-6,-5) arc (180:360:1 and 0.4);
  \draw (5,-0.5) circle (2cm)node[ scale=2]{$A_3^{-}$};
  \draw (3,-0.5) arc (180:360:2 and 0.6);
  \draw (5,-9.5) circle (2cm)node[ scale=2]{$A_3^{+}$};
  \draw (3,-9.5) arc (180:360:2 and 0.6);
  \fill[fill=black] (-2,-5) circle (7pt);
  \draw (-1,-2.2)--(-1,-7.8);
\draw (1,-2.2)--(1,-7.8);
\end{tikzpicture}

\caption{$a_1, a_2$, and $a_3$ are depicted in green, blue, and red, respectively. }

\end{figure}
\FloatBarrier
Let $f_{i,j}^{\prime}:N_{i,j}^{\prime}\xrightarrow{}N_{i,j}^{\prime}$ be defined by: $$ f_{i,j}^{\prime}(x,y,e^{2\pi i\theta}):= \left(x,y, e^{2\pi \textrm{i}\left[\theta+\psi\left(\sqrt{x^{2}+y^{2}}\right)\right]}\right) .$$

We define  $F_{i,j}$  by:

\[   F_{i,j}(p)=\left\{
\begin{array}{ll}
   \phi\circ f_{i,j}^{\prime}\circ\phi^{-1}(p) & \text{ if } p\in N_{i,j} \\
     p & \text{ otherwise. }\\
      
\end{array}
\right. \]

To construct $G_j$, we select a neighborhood $N_j$ diffeomorphic to an open ball minus two open balls with boundaries $A_j^{+}$ and $A_{j}^{-}$ as in the picture below.

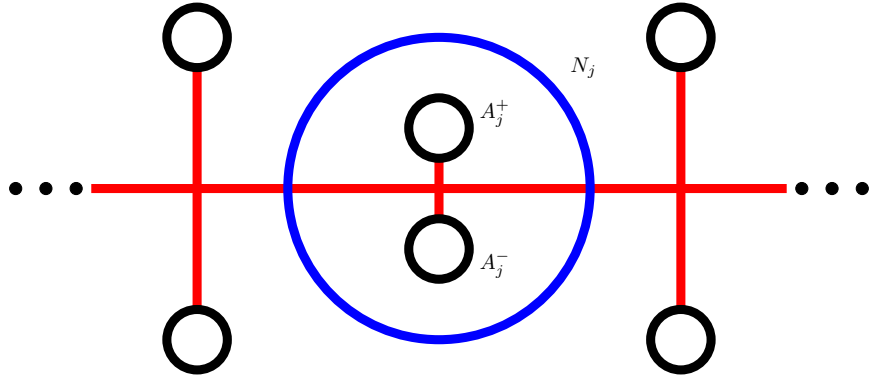
\begin{figure}[h]
\centering

\begin{tikzpicture}[scale=0.4, transform shape]
\draw[red, line width=1.2mm] (-11.5,0) -- (11.5,0);
\draw[red,line width=1.2mm] (-8,4) -- (-8,-4);
\draw [line width=1.2mm] (-8,5) circle (1cm)node[anchor=west, scale=2]{};
\draw[line width=1.2mm] (-8,-5) circle (1cm)node[anchor=west, scale=2]{};
\draw[red, line width=1.2mm] (8,4) -- (8,-4);
\draw[line width=1.2mm] (8,5) circle (1cm)node[anchor=west, scale=2]{};
\draw[line width=1.2mm] (8,-5) circle (1cm)node[anchor=west, scale=2]{};

\draw [blue, line width=1.2mm](0,0) circle (5cm) (4,4) node[black, anchor=west, scale=2]{$N_j$};

\draw [red, line width=1.2mm](0,0) -- (0,1);
\draw [line width=1.2mm] (0,2) circle (1cm) (1,2.5) node[anchor=west, scale=2]{$A_j^{+}$};

\draw [red, line width=1.2mm](0,0) -- (0,-1);
\draw [line width=1.2mm] (0,-2) circle (1cm) (1,-2.5) node[anchor=west, scale=2]{$A_j^{-}$};
\draw (12,0) [fill]circle (0.2cm)node[anchor=west, scale=2]{};
\draw (13,0) [fill]circle (0.2cm)node[anchor=west, scale=2]{};
\draw (14,0) [fill]circle (0.2cm)node[anchor=west, scale=2]{};
\draw (-12,0) [fill]circle (0.2cm)node[anchor=west, scale=2]{};
\draw (-13,0) [fill]circle (0.2cm)node[anchor=west, scale=2]{};
\draw (-14,0) [fill]circle (0.2cm)node[anchor=west, scale=2]{};
\end{tikzpicture}

\caption{The neighborhood $N_j$ is shown in blue.}

\end{figure}
\FloatBarrier

Let $\widetilde{N_j}\subseteq\R^3$ be $B^3=\{(x,y,z)|x^2+y^2+z^2\le 1\}$ minus two open balls with boundary spheres $A_j^{+}$ and $A_j^{-}$, such that these spheres are contained in $\{(x,y,z)|x^2+y^2+z^2<1/9\}$, and they are symmetric with respect to a rotation by $\pi$ around the  $z$-axis. Denote by $N_j^{\prime}$  the space obtained from $\widetilde{N_j}$ with $A_j^{+}$ and $A_j^{-}$ identified  and  choose a diffeomorphism $\phi:N_j^{\prime}\xrightarrow{}N_j$ that respects the trivializations. Parametrize $a_j$ in such a way that $a_j(t)\in \{(x,y,z)|x^2+y^2+z^2<1/9\}$ for $1/3\le t \le 2/3$, and $a_j(t)\in \{(x,y,z)|1/9<x^2+y^2+z^2<1\}$ for $0\le t \le 1/3$ and $2/3\le t\le 1$. We also homotope $a_j$ so that $a_j(t)$ lies on the $z$-axis for $t$ as in the last case. This is depicted in the picture below:

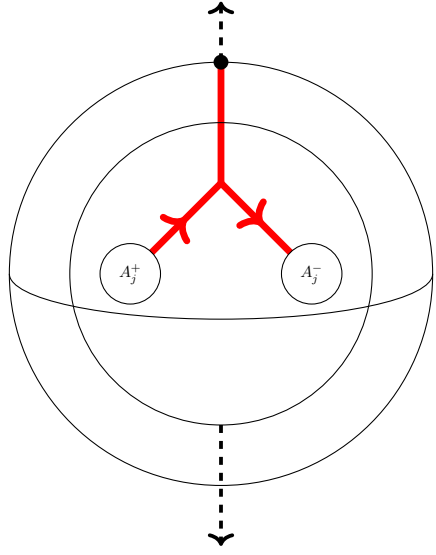
\begin{figure}[ht]
\centering
\begin{tikzpicture}[scale=0.4, transform shape]
\draw[dashed,<-,line width=0.5mm](0,10)--(0,8);
\draw[dashed,->,line width=0.5mm](0,-4)--(0,-8);
\draw[red, line width=0.9mm](0,8)--(0,4) ;

\draw[->-=0.5, red, line width=0.9mm](-2.3,1.7)--(0,4);

\draw[->-=0.6, red, line width=0.9mm](0,4)--(2.3,1.7);

  \draw (-3,1) circle (1cm)node[scale=1.5]{$A_j^{+}$};
  
  \draw (3,1) circle (1cm)node[scale=1.5]{$A_j^{-}$};
  \draw (0,1) circle (5cm);
  \draw (0,1) circle (7cm);
  \draw[](-7,1)arc(180:360:7 and 1.5);
 \fill[fill=black] (0,8) circle (7pt);
  
\end{tikzpicture}
\caption{The loop $a_j$ is depicted in red.}
\end{figure}
 \FloatBarrier

  If $p\in \widetilde{N_j}$ has spherical coordinates $(\theta, \varphi,r)$, define $g_j^{\prime}:\widetilde{N_j}\xrightarrow{}\widetilde{N_j}$ by $g_j^{\prime}(\theta ,\varphi,r):=(\theta +\psi(r)\pi,\varphi,r).$ As $g_j^{\prime}$ respects the identifications of $A_j^{+}$ and $A_j^{-}$, it induces a diffeomorphism on $N_j^{\prime}$ which we still denote by $g_j^{\prime}$.

In \cite{4}, we proved that
 
 \[   G_j(p)=\left\{
\begin{array}{ll}
     \phi\circ g_j^{\prime}\circ \phi^{-1}(p)& \text{ if } p\in N_j  \\
     p & \text{ otherwise, }\\
      
\end{array} 
\right. \]  lifts the inversion $I_j.$

The construction of the  diffeomorphism $\Phi_{(a_i,J_l)}$ lifting $\varphi_{(a_i,J_l)}$ is similar to the construction of $F_{i,j}$. To begin,  select a neighborhood of $a_i,$ denoted by $N,$ diffeomorphic to a  solid torus minus an open ball with boundary $B_l$  as in the picture below:

\begin{figure}[ht]
    \centering
   \includegraphics[width=0.7\textwidth]{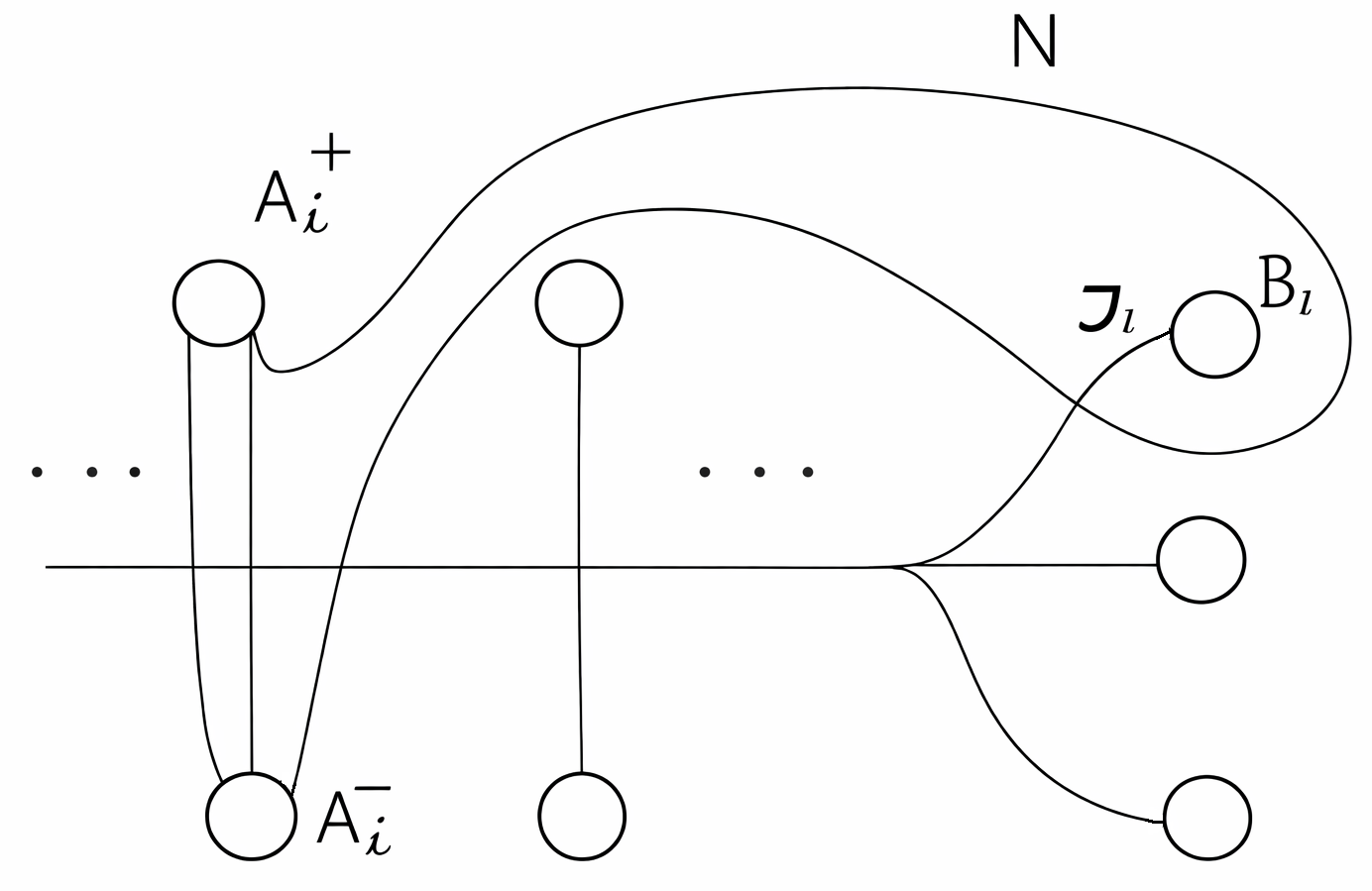}
    \caption{Neighborhood $N$ of $a_i$ in $M_{\Gamma}$.}
\end{figure}
\FloatBarrier

Let $J_l$  be the subinterval of the $l$-th free edge, with endpoints in $B_l$ and $\partial N$.
Choose a diffeomorphism $\phi:N^{\prime}\xrightarrow{}N$ that respects the trivializations, where $N^{\prime}\subseteq\R^3$ is equal to
$D^2 \times S^1$ minus an open ball with boundary $B_l$, which is contained in $\{(x, y)|x^2 + y^2 < 1/9\} \times S^1$, where
$* \times S^1$ is freely homotopic to $a_i$.   This is depicted in the following picture: 


\begin{figure}[h]
\centering

\begin{tikzpicture}[scale=0.45, transform shape]

 \draw (-2,-0.5)--(-2,-9.5);
\draw (2,-0.5)--(2,-9.5);
\draw (-1,-2.2)--(-1,-7.8);
\draw (1,-2.2)--(1,-7.8);
\draw[ blue, line width=1.3mm] (-1,-5)--(-2,-5) (-1.4,-6)node[black, anchor=south, scale=1.5]{$J_l$};
  \draw (0,-0.5) circle (2cm)(1.2,1.4)node[anchor=west, scale=2]{$A_i^{+}$};
  \draw (-2,-0.5) arc (180:360:2 and 0.6);
  \draw (0,-9.5) circle (2cm)(1.2,-11.2)node[anchor=west, scale=2]{$A_i^{-}$};
  \draw (-2,-9.5) arc (180:360:2 and 0.6);
  \draw (0,-5) circle (1cm)(0,-3.5)node[anchor=west, scale=1.5]{$B_l$};
  \draw (-1,-5) arc (180:360:1 and 0.4);
  
  \fill[fill=black] (-2,-5) circle (7pt);
\end{tikzpicture}

\caption{The interval $J_l$ is shown in blue.}

\end{figure}
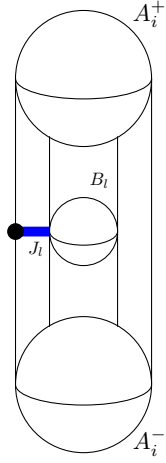
\FloatBarrier

Let $ \phi_{(a_i,J_l)}^{\prime}:N^{\prime}\xrightarrow{}N^{\prime}$ be defined by:

$$ \phi_{(a_i,J_l)}^{\prime}(x,y,e^{2\pi i \theta}):=\left(x,y, e^{2\pi \textrm{i}\left[\theta+\psi\left(\sqrt{x^{2}+y^{2}}\right)\right]}\right).$$

Let $\Phi_{(a_i,J_l)}$ be the map defined  by:

\[   \Phi_{(a_i,J_l)}(p)=\left\{
\begin{array}{ll}
   \phi\circ   \phi_{(a_i,J_l)}^{\prime}\circ \phi^{-1}(p)  & \text{ if } p\in N \\
     p & \text{otherwise.} \\
      
\end{array}
\right. \]  We claim that this map wraps $J_l$ around $a_i$ as shown in the picture below.



\begin{figure}[h]
\centering

\begin{tikzpicture}[scale=0.5, transform shape]

 \draw (-2,-0.5)--(-2,-9.5);
\draw (2,-0.5)--(2,-9.5);
\draw[ blue, line width=1.3mm] (-1,-5)--(-2,-5) (-1.4,-6)node[black, anchor=south, scale=1.5]{};
  \draw (0,-0.5) circle (2cm)(1.2,1.4)node[anchor=west, scale=2]{$A_i^{+}$};
  \draw (-2,-0.5) arc (180:360:2 and 0.6);
  \draw (0,-9.5) circle (2cm)(1.2,-11.2)node[anchor=west, scale=2]{$A_i^{-}$};
  \draw (-2,-9.5) arc (180:360:2 and 0.6);
  \draw (0,-5) circle (1cm)(0,-3.5)node[anchor=west, scale=1.5]{$B_l$};
  \draw (-1,-5) arc (180:360:1 and 0.4);
  \draw (-1,-2.2)--(-1,-7.8);
\draw (1,-2.2)--(1,-7.8);
  \fill[fill=black] (-2,-5) circle (7pt);

  \draw[->-=0.7, red, line width=0.8mm] plot [smooth ] coordinates {(-1.5,-8.2) (-1.5, -6) (-1.6,-5.5) (-2,-5) };
\draw[->-=0.5,red, line width=0.8mm] plot [smooth ] coordinates {(-1,-5) (-1.5, -4) (-1.5,-1.8)};
\end{tikzpicture}

\caption{The interval $J_l$  shown in blue  is sent  by $\Phi_{(a_i, J_l)}$ to the red  arc which is homotopic relative to its endpoints to the path that runs along $J_l$, once around $a_i$, and back along $J_l$.}

\end{figure}
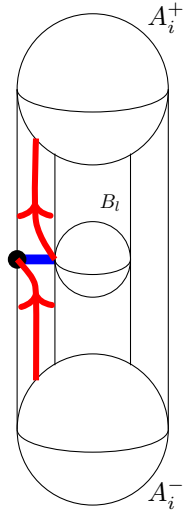
\FloatBarrier

Thus, $\Phi_{(a_i,J_l)}$ lifts $\varphi_{(a_i,J_l)}$.

In order to construct $H$, let $e_1, e_2, e$ and $h$ be as in Section 1. For each piece $M_{e(k)},$ consider the boundary components $B_{k}^{-}$ and $B_k^{+}$ of $M_{e(k)}$ corresponding to $[e(k-1), e(k)]$ and $[e(k),e(k+1)]$. Consider the solid cylinder $N_k$ in $M_{e(k)}$ with boundary $B_{k}^{-}$ and $B_k^{+}$ that contains the identified spheres $A_{k}^{+}$ and $A_k^{-}$ corresponding to $a_k$, as in the picture below.

\begin{figure}[h]
    \centering
    \includegraphics[width=0.65\textwidth]{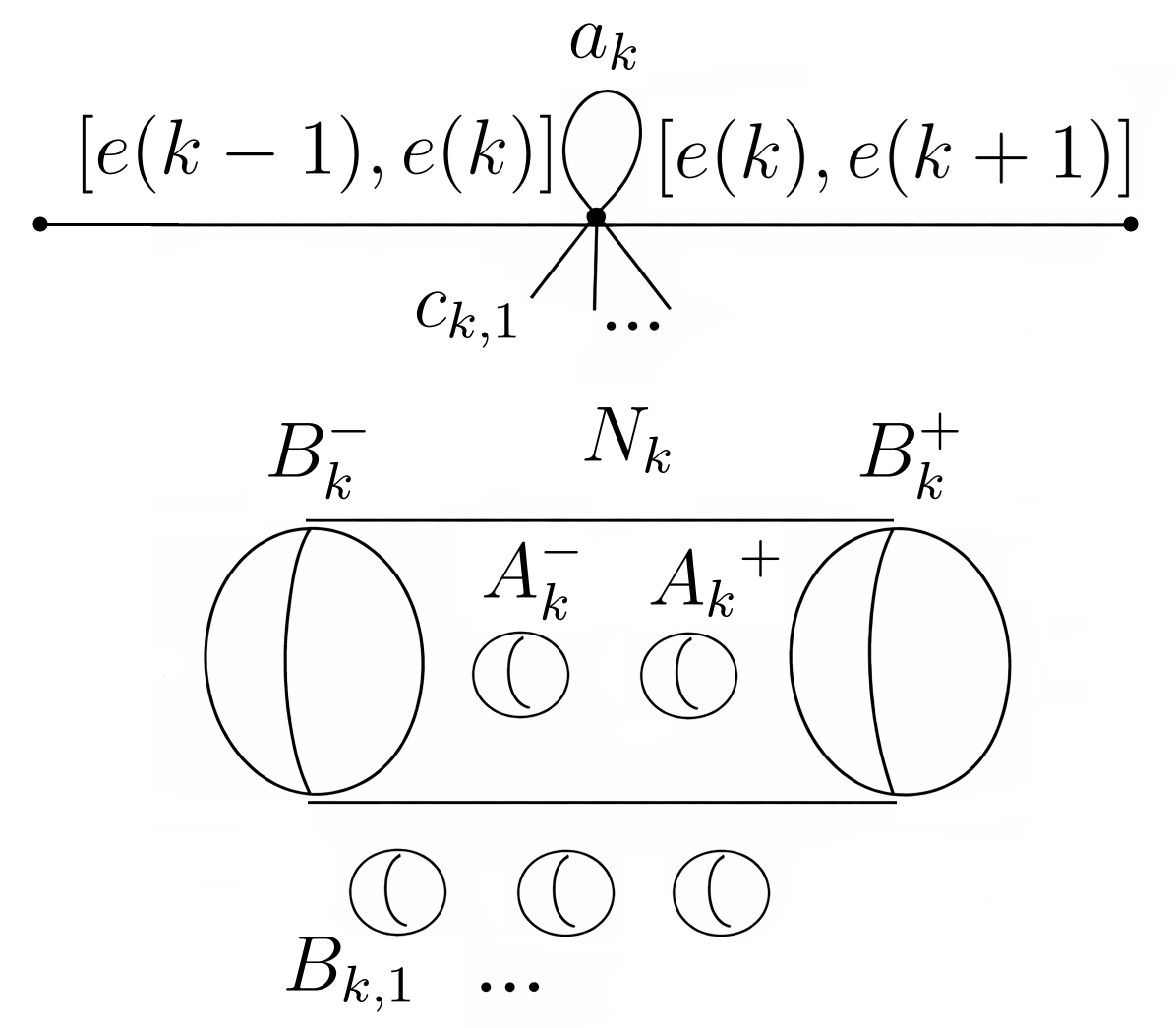}
    \caption{Neighborhood $N_k$ diffeomorphic to a solid cylinder with boundary components $B_k^{-}$ and $B_k^{+}.$}
    \label{fig:placeholder}
\end{figure}
\FloatBarrier

Consider the neighborhood $N$ of $e(\R)$ defined as the union of all the $N_k$, where $B_k^{+}$ is identified with $B_{k+1}^{-}$. Find a diffeomorphism $\phi:N^{\prime}\xrightarrow{}N$ that respects the trivializations, where:  $$N^{\prime}=\{(x,y,z)| y^2+z^2\le 4\}-\{A^{-}_k,A^{+}_k\}_{k\in\Z}\subseteq\R^{3},$$ where  $A^{-}_k:=\{(x,y,z)|(x-2(2k))^2+y^2+z^2<1/16\}$ and $A^{+}_k:=\{(x,y,z)|(x-2(2k+1))^2+y^2+z^2<1/16\}$. $N^{\prime}$ is depicted in the picture below: 

\begin{figure}[h]
    \centering
    \includegraphics[width=0.9\textwidth]{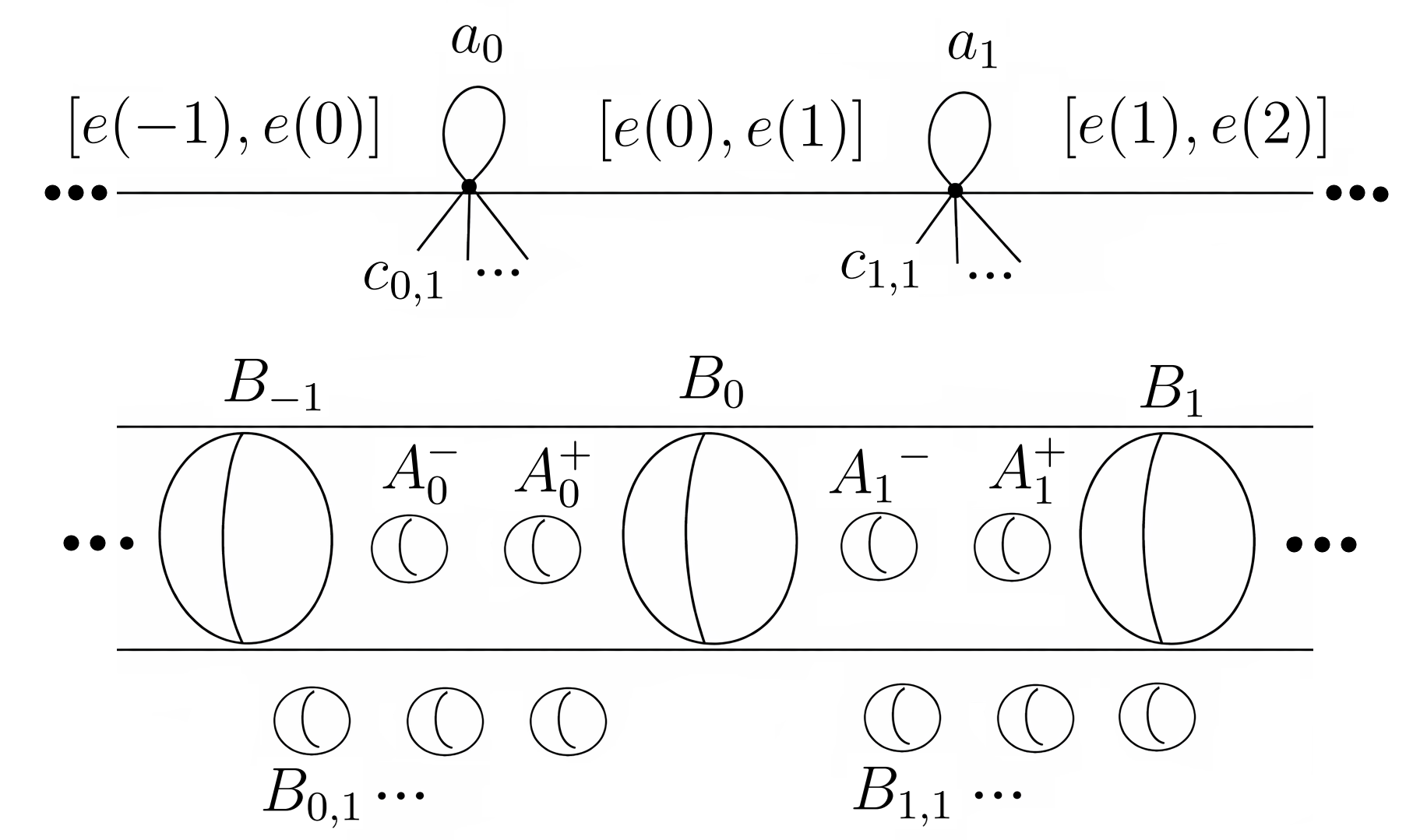}
    \caption{An embedded graph with one loop attached at each vertex and two ends accumulated by loops is shown above. The associated neighborhood $N^{\prime}$ is equal to $D^2\times \R$ minus  an infinite number of open balls $\{A^{+}_k, A^{-}_k\}_{k\in\Z}$, where $A^{+}_k$ is identified with $A^{-}_{k}$.  }
    \label{fig:placeholder}
\end{figure}
\FloatBarrier

Construct a smooth  function $\psi:[0,2]\xrightarrow{} [0,1]$ with $\psi(r)=1$ on [0,1], $\supp(\psi)\subseteq [0,5/4)$, and decreasing, so  that $\psi^{\prime}(r)\le 0$. Denote by $h^{\prime}:N^{\prime}\xrightarrow{}N^{\prime}$ the diffeomorphism given by: 
$$h^{\prime}(x,y,z)=(x+4\psi(\sqrt{y^2+z^2}),y,z).$$

Let $H$ be the map defined by:

 \[   H(p)=\left\{
\begin{array}{ll}
       \phi\circ h^{\prime}\circ\phi^{-1}(p) & \text{ if } p\in N  \\
     p & \text{ otherwise. } \\
      
\end{array} 
\right. \]

It is easy to see that $H$ is a diffeomorphism that lifts $h$.

The setup for constructing the  diffeomorphism $P_{i,j}$ lifting $p_{i,j}$ is analogous to the setup used in the construction of $G_j$. It suffices to consider the case where $p_{i,j}$ swaps the ends accumulated by loops, as the other case can be treated more easily. To begin, choose a neighborhood $N$ of $e_i \cup e_j$ containing neighborhoods $N_i$ and $N_j$ of the two ends, each one diffeomorphic to the connected sum of infinitely many copies of $ S^2\times S^1$ minus a ball, as in the picture below:

\begin{figure}[h]
    \centering
   \includegraphics[width=0.45\textwidth]{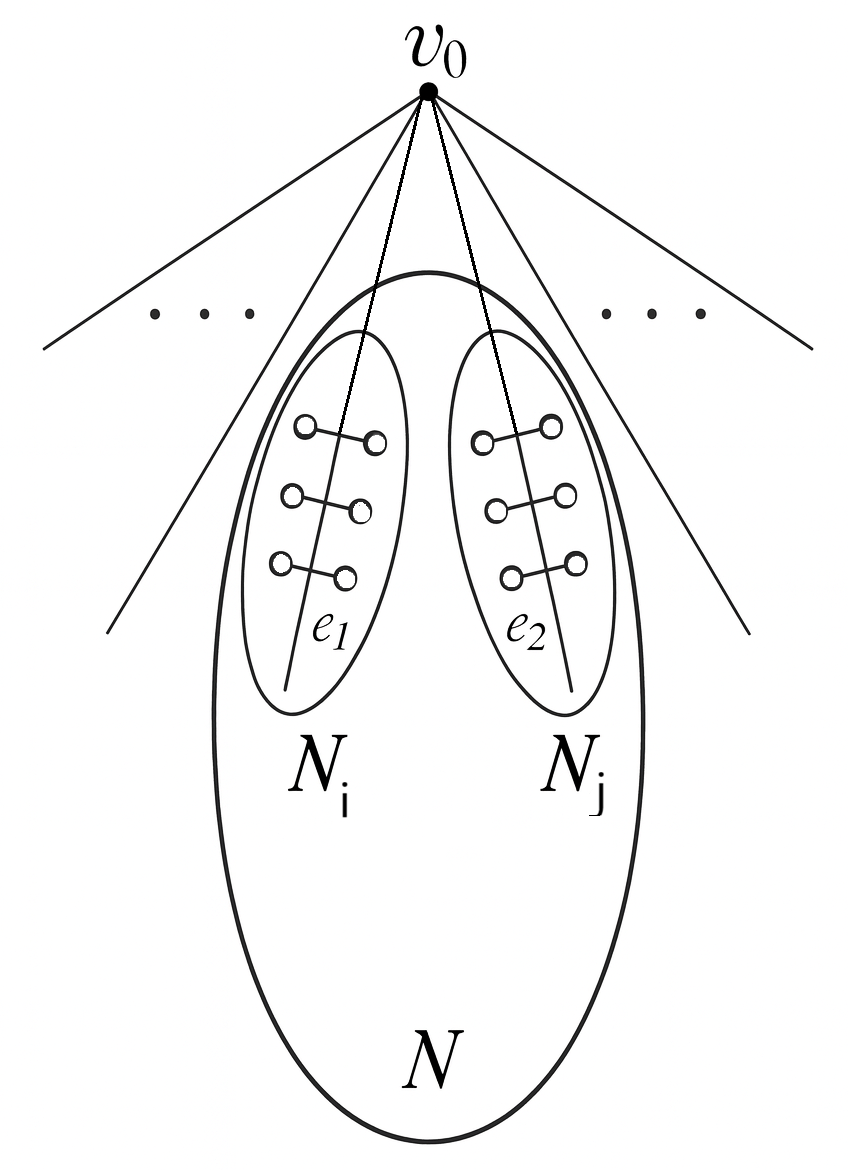}
    \caption{ Neighborhood $N$ of $e_i\cup e_j$ in $M_{\Gamma}$.}
    \label{fig:placeholder}
\end{figure}
\FloatBarrier

   Let $B^3=\{(x,y,z)\mid x^2+y^2+z^2\le 1\}$, and let $D_i,D_j\subset B^3$ be
disjoint open balls contained in $\{(x,y,z)\mid x^2+y^2+z^2<1/9\}$ that are
exchanged by the rotation by $\pi$ about the $z$-axis. Let $N'$ be the space
obtained from $B^3\setminus(D_i\cup D_j)$ by gluing copies $N'_i$ and $N'_j$
of $N_i$ and $N_j$ along $\partial D_i$ and $\partial D_j$, as depicted in the
picture below.


\begin{figure}[h]
\centering
    
\begin{tikzpicture}[scale=0.5, transform shape]

 \draw[dashed,<-,line width=0.5mm](0,10)--(0,8);\draw[dashed,->,line width=0.5mm](0,-6)--(0,-8);



  \draw (-3,1) circle (1cm)node[scale=1.5]{$N_i^{\prime}$};
  
  \draw (3,1) circle (1cm)node[scale=1.5]{$N_j^{\prime}$};
 \draw (0,1) circle (5cm)(0,3)node[scale=3.5]{$N^{\prime}$};
  \draw (0,1) circle (7cm); 
  \draw[](-7,1)arc(180:360:7 and 1.5);

   \fill[fill=black] (0,8) circle (7pt);
\end{tikzpicture}

\caption{The model $N^{\prime}$: the unit ball with $N_i^{\prime}$ and $N_j^{\prime}$ glued in, exchanged by the rotation by $\pi$ about the $z$-axis}.
\end{figure}
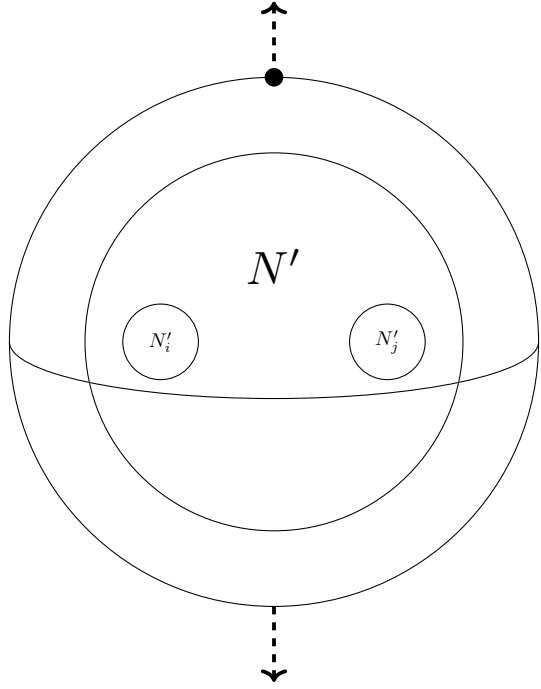
\FloatBarrier
 Choose a diffeomorphism $\phi\colon N'\to N$ that respects the
trivializations and sends $N'_i$ and $N'_j$ to $N_i$ and $N_j$. Choose a diffeomorphism $\nu:N_i^{\prime}\xrightarrow{}N_{j}^{\prime}$ that is compatible with the rotation by $\pi$ on a collar neighborhood of $\partial N_i^{\prime}\cup \partial N_j^{\prime}$ and that preserves the trivialization up to homotopy. If $p\in N^{\prime} -(N_i^{\prime}\cup N_j^{\prime})$ has spherical coordinates $(\theta, \varphi,r)$, let $p_{i,j}^{\prime}:N^{\prime}\xrightarrow{}N^{\prime}$ be the map defined by: 

\[   p_{i,j}^{\prime}(p)=\left\{
\begin{array}{ll} 
   \nu(p) & \text{if } p\in N_i^{\prime}\\
   \nu^{-1}(p) & \text{if } p\in N_j^{\prime}\\
       (\theta+\psi(r)\pi, \varphi, r)&  \text{if }p\in N^{\prime} -(N_i^{\prime}\cup N_j^{\prime}).  \\

\end{array} 
\right. \]

 Then, the map $P_{i,j}$  defined by:
 \[   P_{i,j}(p)=\left\{
\begin{array}{ll} 
   \phi\circ p_{i,j}^{\prime}\circ\phi^{-1}(p) & \text{if } p\in N\\
  
     p & \text{otherwise,} \\
      
\end{array}
\right. \]   lifts the end swap $p_{i,j}.$

\section{Twisting crossed homomorphism for a manifold with boundary}\label{sec:Twisting crossed}

Consider again the short exact sequence:

$$ 1\xrightarrow{} \ts(M_{n,k})\xrightarrow{}\map(M_{n,k})\xrightarrow{\rho_{n,k}}\map({{G}_{n,k}})\xrightarrow{}1.$$ Notice that when $k=0$, we get Laudenbach's short exact sequence:

 $$1 \xrightarrow{} \ts(M_n)\xrightarrow{} \map(M_n)
\xrightarrow{\rho} \Out(F_n) \xrightarrow{}1.$$

We will prove that the Hatcher–Vogtmann sequence splits for $k \ge 1$ and provide a section $s_{n,k}:\map({G_{n,k}})\xrightarrow{}\map(M_{n,k})$ for this short exact sequence. In order to do this, we study the derivative  and twisting crossed homomorphisms introduced by Brendle--Broaddus--Putman in \cite{3},  in the case where the manifold has a boundary. 
We first recall the closed case from \cite{3}. Let $M$ be a closed oriented 3-manifold.
We will now define the derivative crossed homomorphism $$\mathcal{D}: \diff(M)\xrightarrow{} C(M,\GLR).$$

Given  $F\in\diff(M)$, the derivative crossed homomorphism evaluated at $[F]$, $\mathcal{D}([F]):M\xrightarrow{}\GLR$, gives for each $p$ a linear transformation $\mathcal{D}([F])(p)$ in $\GLR$, obtained  as follows.  It is the unique linear transformation that makes the following diagram commute:

$$ $$

\[ \begin{tikzcd}[transform canvas={scale=1.3}, row sep=25,column sep = 40]
\R^{3} \arrow{r}{\sigma_0(p)} \arrow[swap]{d}{\mathcal{D} ([F])(p)} & T_pM \\
\R^{3} \arrow{r}{\sigma_0(F(p))}& T_{F(p)}M \arrow{u}[swap] {[DF^{-1}]_{F(p)}}
\end{tikzcd}
\]

$$ $$

 Thus, $\mathcal{D}([F])(p)$ is the inverse of the linear transformation that represents the change of basis transformation from the basis

$$\{\sigma_0(p)(e_1), \sigma_0(p)(e_2),\sigma_0(p)(e_3)\}$$ of $T_{p}M$ to the basis  

$$\{DF^{-1}(\sigma_0(F(p))(e_1)), DF^{-1}(\sigma_0(F(p))(e_2)), DF^{-1}(\sigma_0(F(p))(e_3))\}$$  of $T_pM$, as depicted in the figure below.

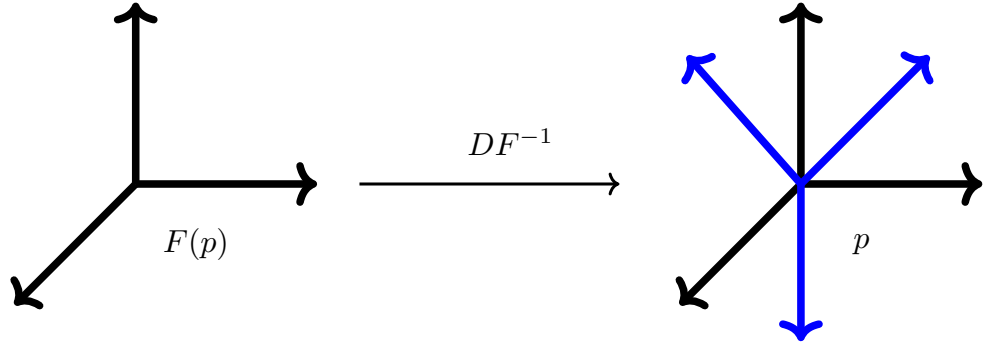
\begin{figure}[hbt!]
\centering
\begin{tikzpicture}[scale=0.8, transform shape]
    \draw[->,line width=1mm](-3,0)--(-5,-2);
    \draw[->, line width=1mm](-3,0)--(-3,3);
    \draw[->, line width=1mm](-3,0)--(0,0);
     \draw[->,line width=0.4mm](0.7,0)--(5,0);
\draw[->,line width=1mm](8,0)--(6,-2);
    \draw[->,line width=1mm](8,0)--(8,3);
    \draw[->,line width=1mm](8,0)--(11,0);
\draw[->,line width=1mm, blue](8,0)--(10.12,2.12);
\draw[->,line width=1mm, blue](8,0)--(6.12,2.12);
\draw[->,line width=1mm, blue](8,0)--(8,-2.6);
\draw (-2,-1) node[scale=1.5]{$F(p)$};
 \draw (9,-1) node[scale=1.5]{$p$};
\draw (3.2,0.7) node[scale=1.5]{$DF^{-1}$};
\end{tikzpicture}
\caption{The bases at $F(p)$ and $p$ that are determined by $\sigma_0$ are depicted in black. The  basis at $F(p)$  is sent to the basis in blue at $p$ by $DF^{-1}$. $\mathcal{D}([F])^{-1}(p)$ is the change of basis between the blue and black bases at $p$ }.
\end{figure}
\FloatBarrier
Therefore, we get:
$$\mathcal{D}([F])^{-1}(p)=\sigma_0^{-1}(p)[DF^{-1}]_{F(p)}\sigma_0(F(p)).$$

Now, we define the twisting crossed homomorphism: $$\mathfrak{T}:\map(M)\xrightarrow{}\Hom(\pi_1(M,x_0), \pi_1(\GL^{+}(3,\R), \mathrm{Id}))=H^{1}(M, \Z_2).$$ The twisting crossed homomorphism evaluated at $[F]$, $\mathfrak{T}([F]):\pi_1(M,x_0)\xrightarrow{}\pi_1(\GL^{+}(3,\R), \mathrm{Id})$,   is the homomorphism that sends $[\gamma]\in \pi_1(M,x_0)$ to the class $\mathfrak{T}([F])[\gamma]=[\mathcal{D}([F])(\gamma)]$, where  $\mathcal{D}([F])(\gamma)$ is the loop:

$$[0,1]\xrightarrow{\gamma}M\xrightarrow{\mathcal{D}([F])}\GL^{+}(3,\R). $$ In other words, $\mathfrak{T}([F])$ is the map induced by $\mathcal{D}([F])$ on fundamental groups.

Let $M$ be a manifold with nonempty boundary. Let $F:M\xrightarrow{}M$ be a diffeomorphism  that is the identity on the boundary. Since $M-\partial M$ is a smooth manifold, we get a derivative crossed homomorphism: 
$$\mathcal{D}:\diff(M-\partial M)\xrightarrow{}C(M-\partial M,\GLR).$$
Isotope $F$ such that it is the identity in a normal neighborhood $N(\partial M)$ of $\partial M.$ Then $\mathcal{D}$ extends to a map: 
$$\mathcal{D}:\diff(M,\partial M)\xrightarrow{}C((M,\partial M),(\GLR,\mathrm{Id})).$$ 
We call this map the relative derivative crossed homomorphism, which descends to a map: 

$$\mathcal{D}:\map(M)\xrightarrow{}[(M,\partial M),(\GLR,\mathrm{Id})].$$


    

We now use the lemma below to construct the relative twisting crossed homomorphism.

\begin{lema}\label{lem:functor}
Let $\mathcal{F}$ be a functor from the category of pairs of topological spaces to the category $Gr$ of groups satisfying: 

\begin{itemize}
    \item  Homotopy invariance: $\mathcal{F}(f)$ depends on the homotopy class of $f$, so we get a map: $\mathcal{F}: [ (M, \partial M),  (\GLR, \mathrm{Id})]\xrightarrow{}\Hom(\mathcal{F}(M, \partial M),\mathcal{F}(\GLR, \mathrm{Id})) $.

    \item  Compatibility with pointwise products: For maps $g, h: (M, \partial M) \xrightarrow{} (\GLR, \mathrm{Id})$, let $g\cdot h$ be the pointwise product $p\xrightarrow{} g(p)h(p)$. Then $\mathcal{F}(g\cdot h) = \mathcal{F}(g) · \mathcal{F}(h)$.
\end{itemize}

 Then, the map $$ \mathcal{F}\circ\mathcal{D}: \map(M)\xrightarrow{}\Hom(\mathcal{F}(M,\partial M),\mathcal{F}(\GLR,\mathrm{Id}))$$
  is a crossed homomorphism. 

    \begin{proof}
        Since $\mathcal{D}:\map(M)\xrightarrow{}[(M,\partial M),(\GLR,\mathrm{Id})]$ is a crossed homomorphism, we have that $\mathcal{D}(fg) = \mathcal{D}(f)^g\cdot \mathcal{D}(g)$, where $\mathcal{D}(f)^g=\mathcal{D}(f)\circ g$. Applying  $\mathcal{F}$ we get $\mathcal{F}(\mathcal{D}(fg)) = \mathcal{F}(\mathcal{D}(f) \circ g) \cdot \mathcal{F}(\mathcal{D}(g))$ because $\mathcal{F}$ is compatible with products. Functoriality gives $\mathcal{F}(\mathcal{D}(f) \circ g) = \mathcal{F}(\mathcal{D}(f)) \circ \mathcal{F}(g)=((\mathcal{F}\circ \mathcal{D})(f))^{\mathcal{F}(g)}.$ Thus, $\mathcal{F}\circ\mathcal{D}(fg)=((\mathcal{F}\circ \mathcal{D})(f))^{\mathcal{F}(g)}\cdot ((\mathcal{F}\circ\mathcal{D})(g))$.
    \end{proof}
\end{lema}

We note that the functor $H_1$ is compatible with pointwise products because $\GLR$ is a topological group. Thus,  we get a well-defined map

$$[(M,\partial M),(\GLR,\mathrm{Id})]\xrightarrow{}\Hom(H_1(M,\partial M),H_1(\GLR,\mathrm{Id}))\cong$$
$$\Hom(H_1(M,\partial M),\Z_2),$$ and hence a map

$$\tsh:\map(M)\xrightarrow{}\Hom(H_1(M,\partial M),\Z_2)\cong H^1(M,\partial M).$$

 We call this map the relative twisting crossed homomorphism. In the case where $M=M_{n,k}$, we denote $\tsh$ by $\tsh_{n,k}$.

 \begin{remark}
When $\partial M=\emptyset$, Lemma~\ref{lem:functor} with $\mathcal{F}=H_1$ gives a crossed
homomorphism $\mathrm{Map}(M)\to\mathrm{Hom}(H_1(M),\mathbb{Z}_2)=H^1(M)$.
Since $\pi_1(GL^+(3,\mathbb{R}),\mathrm{Id})\cong\mathbb{Z}_2$ is abelian, every
homomorphism $\pi_1(M,x_0)\to\mathbb{Z}_2$ factors through $H_1(M)$, so this
map coincides with the twisting crossed homomorphism $\tsh$ defined
above. Thus the relative twisting crossed homomorphism extends $\tsh$
to manifolds with boundary.
\end{remark}

We now develop the tools needed to prove that $\ts(M_{n,k})$ is isomorphic to $H^1(M_{n,k},\partial M_{n,k})$.

\begin{prop}\label{prop generating set of H(Mnk)}

Let $k\ge 1$. Then we have:
$$H^1(M_{{n,k}},\partial M_{{n,k}})\cong\Z_2^{n+k-1}.$$
    
\begin{proof}

Consider the long exact sequence in relative homology for $M_{n,k}$:

   \[\begin{multlined}\cdots\xrightarrow{} H_1(\partial M_{{n,k}})\xrightarrow{}H_1(M_{{n,k}})\xrightarrow{}H_1(M_{{n,k}},\partial M_{{n,k}})\xrightarrow{}\\\widetilde{H}_0(\partial M_{{n,k}})\xrightarrow{}\widetilde{H}_0(M_{{n,k}})\xrightarrow{}\cdots.\end{multlined}\]

Since $H_1(\partial M_{{n,k}})\cong 0, H_1(M_{{n,k}})\cong \Z_2^n, \widetilde{H}_0( \partial M_{{n,k}})\cong \Z_2^{k-1}$ and $\widetilde{H}_0(M_{{n,k}})\cong 0$,  we obtain that $H_1(M_{{n,k}},\partial M_{{n,k}})\cong \Z_2^{n+k-1}$.
Then, $$H^1(M_{{n,k}},\partial M_{{n,k}})=\Hom(H_1(M_{{n,k}},\partial M_{{n,k}}),\Z_2)\cong \Z_2^{n+k-1}.$$
\end{proof}
\end{prop} 
A basis for $H_1(M_{{n,k}},\partial M_{{n,k}})$ consists of the images of the generators $a_i$ of the fundamental group of $M_{n,k}$ for  $1\le i\le n$, and a choice of $k-1$ independent relative cycles. A standard choice of relative cycles in the basis that will be convenient in Section \ref{sec:Computation twisting crossed} of this paper is given by the set $\{b_1,...,b_{k-1}\}$, where $b_s$ is the simple arc with endpoints in $B_s$ and $B_{s+1}$ as in the picture shown below:

\begin{figure}[h]
    \centering
  \includegraphics[width=0.67\textwidth]{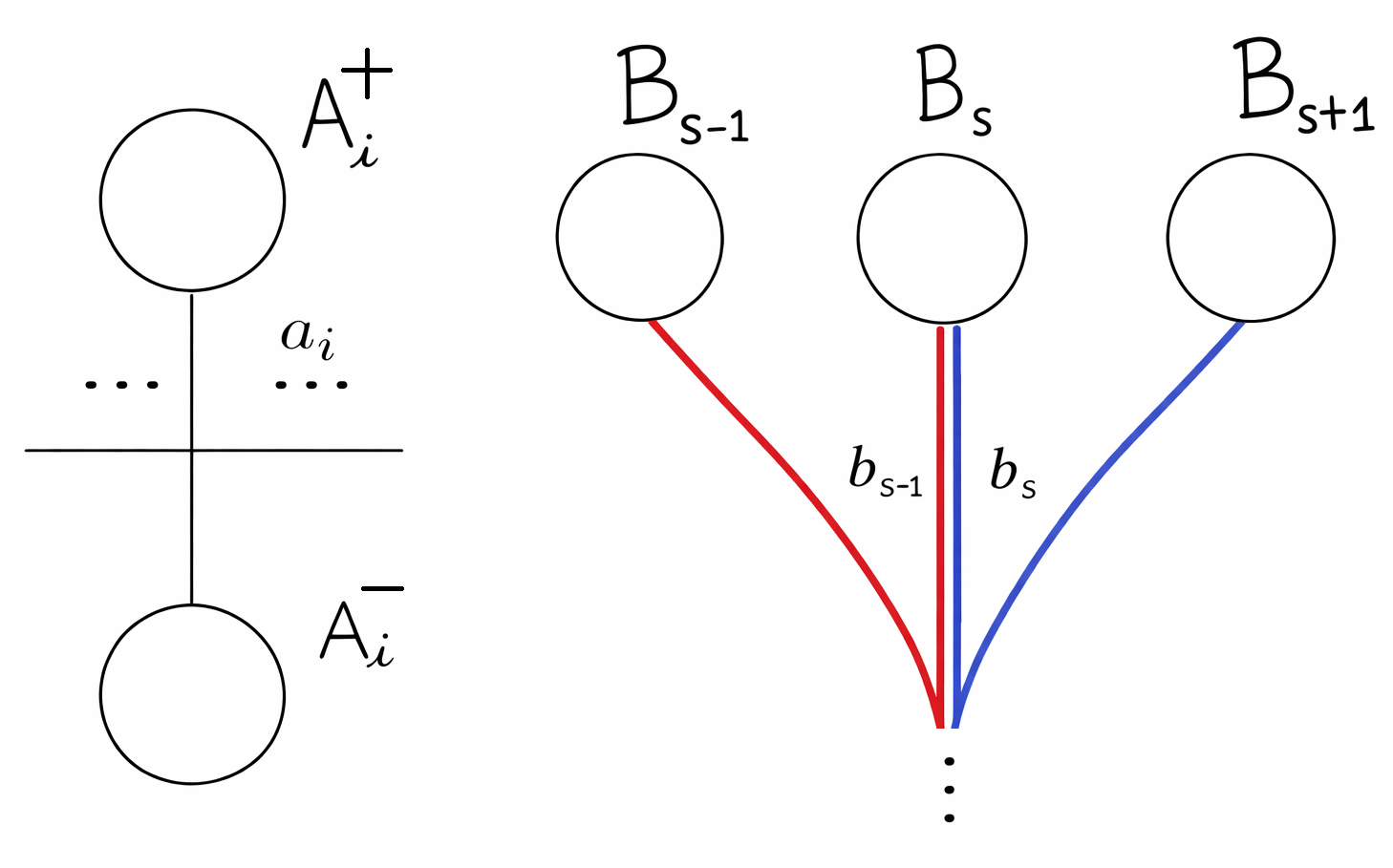}
    \caption{ By $a_i$ we denote the generators corresponding to the fundamental group of $M_{n,k}$. By $b_s$ we denote the generators which are the simple arcs with endpoints in $B_s$ and $B_{s+1}.$  }
    \label{fig:placeholder}
\end{figure}
\FloatBarrier

For Proposition \ref{prop isomorphism Twist and cohomology} we use a different basis for $H_1(M_{{n,k}},\partial M_{{n,k}})$. This is the set $\{\beta_2,...,\beta_k\}$, where $\beta_s$ is the simple arc with endpoints in $B_s$ and $B_{1}$ as in the picture shown below:

\begin{figure}[h]
    \centering
  \includegraphics[width=0.67\textwidth]{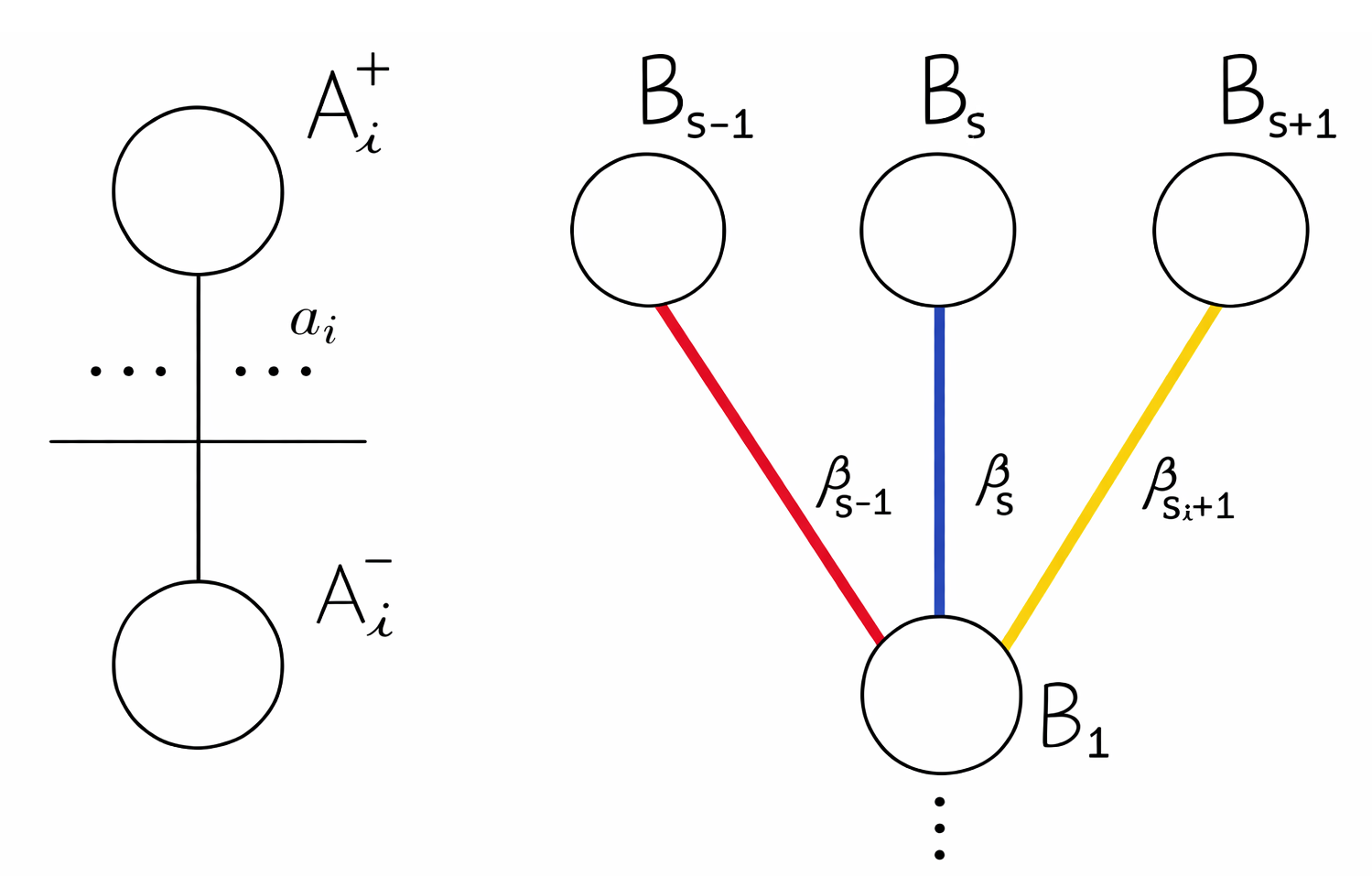}
    \caption{ By $a_i$ we denote the generators corresponding to the fundamental group of $M_{n,k}$. By $\beta_s$ we denote the generators which are the simple arcs with endpoints in $B_s$ and $B_{1}.$ }
    \label{fig:placeholder}
\end{figure}
\FloatBarrier

For the reader's convenience, we adapt the computations from \cite[pp. 214--215]{18} in the following lemma. 
\begin{lema}\label{lemma product of twist trivial}
   In $M_{n,k}$ we have that  $[T_{B_1}]\cdots [T_{B_k}]=\mathrm{Id}$, where $B_1,...,B_k$ are the boundary spheres of $M_{n,k}$.

 \begin{proof}
  Let $X$ be the manifold obtained from $M_{n,k}$ by cutting along the $n$ core spheres $S_1,\ldots,S_{n}$. Then $X$ is a 3-ball with $2n+k-$1 disjoint sub-balls removed. Align the $2n+k-$1 boundary spheres along the $z$-axis, and isotope $X$ such that it is invariant under rotations by any angle about the $z$-axis, as shown in the figure below.  
    
\begin{figure}[ht]
    \centering
     \includegraphics[width=0.4\textwidth]{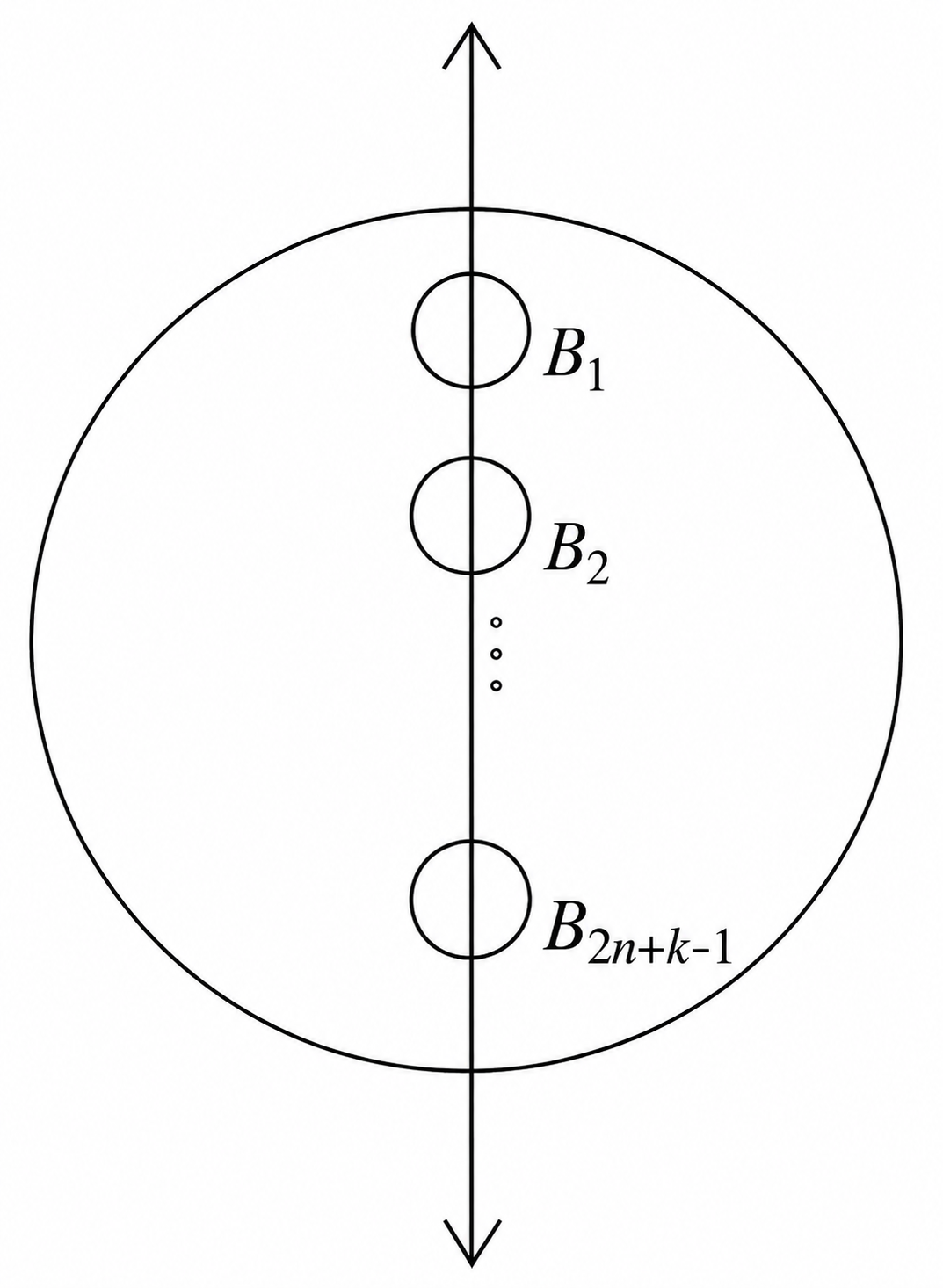}
    \caption{Manifold $S^3$ with $2n+k$ open balls removed.}
\end{figure}
\FloatBarrier

Choose neighborhoods $U_i\cong B_i\times [0,1]$ of the boundary spheres $B_i,$ as shown in red in the picture below:

\begin{figure}[ht]
    \centering
    \includegraphics[width=0.45\textwidth]{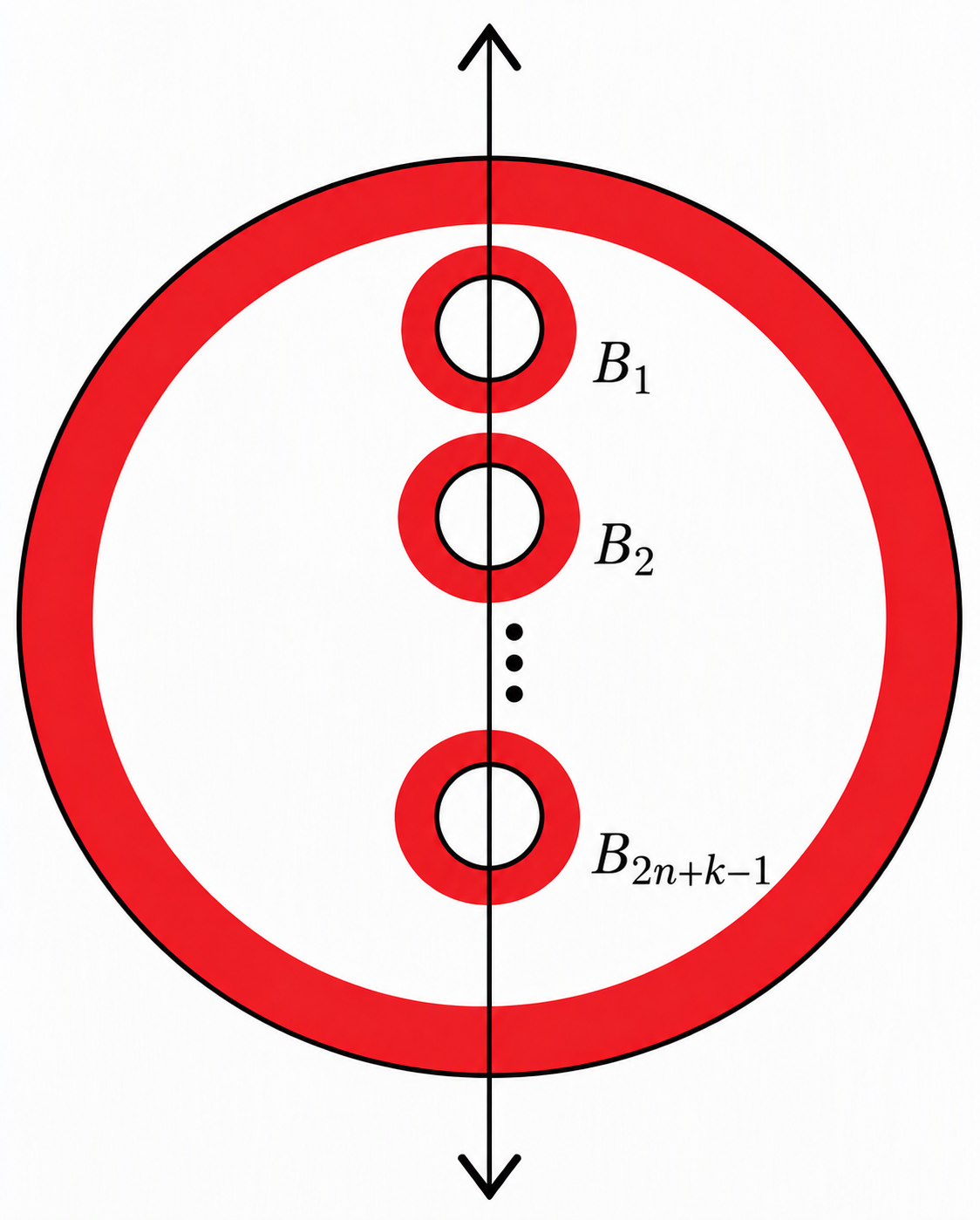}
    \caption{Neighborhoods $U_i$ of $B_i$ shown in red.}
\end{figure}
\FloatBarrier

Let $r_{\theta}:X\xrightarrow{}X$ be rotation about the $z$-axis by $2\pi\theta.$ Let $H:X\times [0,1]\xrightarrow{}X$ be the map defined by:

  \[   H(x,\tau)=\left\{
\begin{array}{ll}
(r_{\tau t}(s),t) & \text{ if } x=(s,t)\in U_i \\
     r_\tau(x) &   \text{ otherwise}. \\
      
\end{array}
\right. \]

To see that $H$ is continuous, observe that if $x=(s,1)\in B_i\times\{1\},$ then $(r_{\tau }(s),1)=r_{\tau}(s,1)=r_\tau(x).$
Moreover, $H(x,0)=r_0(x)=x$ and $H(-,1)=T_{B_1}\cdots T_{B_{2n+k}}$.
Thus, $H$ is an isotopy from  $\text{Id}$ to the map $T_{B_1}\cdots T_{B_{2n+k}}$. Since $H$ fixes the boundary of $X$ pointwise, it induces a map $H: M_{n,k}\times [0,1]\xrightarrow{}M_{n,k}$. Since $2n$ sphere twists descend to pairwise identical sphere twists about the core spheres of $M_{n,k}$, and since these twists commute and have order at most two, we conclude that $H$ is an isotopy from $\text{Id}$ to $T_{B_1}\cdots T_{B_{k}}$.
     
  
   \end{proof}
\end{lema}

The proof of the next proposition follows the lines of Claim 2 in the proof of Theorem 2.1  in \cite{3}.
\begin{prop}\label{prop Twist generated by TS and TB}
    $\ts(M_{{n,k}})$ is generated by the sphere twists about the core spheres of the $n$ summands of $S^2\times S^1$ in $M_{n,k}$ and by the sphere twists about $k-1$ boundary components of $M_{n,k}$. Therefore: 
    $$\ts(M_{{n,k}})= \langle [T_{S_1}],...,[T_{S_n}],[T_{B_2}],...,[T_{B_k}]\rangle.$$

    \begin{proof}
        Let $X$  be the 3-manifold with boundary obtained by removing $2n+k$ disjoint open balls from $S^3$. By \cite[Lemma 2.5]{3},  $\map(X)$ is generated by the set $\mathcal{S}$ consisting of $2n+k$ sphere twists about embedded spheres that are parallel to components of $\partial X$. By Claim 2 of the proof of Theorem 2.1 in \cite{3}, we see that the map $\map(X) \xrightarrow{} \map(M_{n,k})$ induced by gluing has image $\ts(M_{{n,k}})$. Thus, the image of a generating set of $\map(X)$ is a generating set of $\ts(M_{{n,k}})$. The image of $\mathcal{S}$ in $\map(M_{{n,k}})$ is the set of sphere twists about the core spheres of the $n$ summands of $S^2\times S^1$ in $M_{n,k}$ and  the sphere twists about the $k$ boundary components of $M_{n,k}$. Since the product of all the boundary twists is trivial by Lemma \ref{lemma product of twist trivial},  we have that $[T_{B_1}]=\left[\prod_{j=2}^{k}T_{B_j}\right].$ This implies that
    $ \langle [T_{B_1}],...,[T_{B_k}]\rangle= \langle [T_{B_2}],...,[T_{B_k}]\rangle.$
       
    Thus: $$\ts(M_{{n,k}})= \langle [T_{S_1}],...,[T_{S_n}],[T_{B_2}],...,[T_{B_k}]\rangle.$$
          \end{proof}
\end{prop}

We have the following proposition:
\begin{prop}\label{prop isomorphism Twist and cohomology}
 
 Let $S$ be one of the core spheres $S_i,$ $1\le i\le n$, in $M_{n,k}$ or one of the boundary components $B_i,$ $2\le i\le k$. Then,  $\tsh_{n,k}([T_{S}])\in H^1(M_{{n,k}},\partial M_{{n,k}})$ is the cohomology class that is 1 when evaluated on $a_i$ or $\beta_i$ if $S=S_i$ or $S=B_i$ respectively, and 0 otherwise. Moreover: 
 
 $$\ts(M_{{n,k}})\cong H^1(M_{{n,k}},\partial M_{{n,k}})\cong \Z_{2}^{n+k-1}.$$
  \begin{proof}
      When $S_i$ is one of the core spheres of $M_{n,k}$, the proof follows as in \cite[Lemma~5.1]{3}, and the boundary case is the same argument applied to the arcs 
      $\beta_i$. We review the main step here. Consider the normal neighborhood $S_i\times [0,1]$ of $S_i$ and let $J$ be the subpath of $a_i$ which is the intersection of $S_i\times [0,1]$ with $a_i$ as in the picture below.

\begin{figure}[h]
\centering

\begin{tikzpicture}[scale=0.45, transform shape]

 \draw [blue, line width=1.5mm] (0,3)--(0,5) (0,4)node[black, anchor=west, scale=2]{$J$};
 \draw[red, line width=1.3mm] (0,-3)--(0,3) (0,0)node[black, anchor=west, scale=2]{$a_i$};
 \draw (0,7) circle (2cm) (2,7)node[anchor=west, scale=2]{$S_i^{+}$};
\draw (0,7) circle (4cm) (4,8.5) node[anchor=west, scale=2]{$S_i^{+}\times [0,1]$};
\draw (0,-5) circle (2cm) (2, -3.5) node[anchor=west, scale=2]{$S_i^{-}$};

\end{tikzpicture}

\caption{The interval $J$ in blue is homotoped so it is included in the axis of rotation of $T_{S_i}$.}

\end{figure}
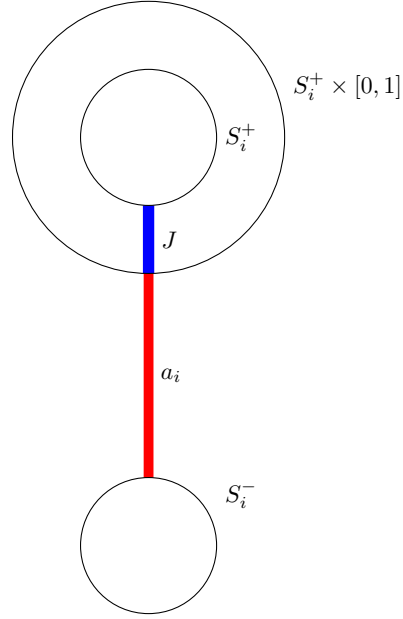
\FloatBarrier
    
     Let $\ell(t)$ be the generator of $\pi_1(\GLR, \mathrm{Id})$. Recall that $T_{S_i}(s,t)=(\ell(t)s,t)$ if $(s,t)\in  S_i\times [0,1]$, and $T_{S_i}$ is the identity otherwise. Since $a_i$ can be homotoped so that $J$ is homotopic to $p_0\times [0,1]$, where $p_0$ is the south pole,  $\mathcal{D}([T_{S_i}])(p_0,t)$ is a rotation by an angle of $2\pi t$. 
     Hence $\mathcal{D}([T_{S_i}])$ is homotopic to $\ell(t)$ along 
     $J$ and is the identity outside the collar, so for each generator $c_j$ of $H_1(M_{{n,k}},\partial M_{{n,k}})$ we have:

      \[   \tsh_{n,k}([T_{S_i}])[c_j]=\left\{
\begin{array}{ll}
1 & \text{ if } c_j=a_i \\
     0 &   \text{ otherwise}. \\
      
\end{array}
\right. \]

Now, let $S=B_i$ be one of the boundary components of $M_{n,k}$, and consider the neighborhood $B_i\times [0,1]$  of $B_i$ as in the picture below:

 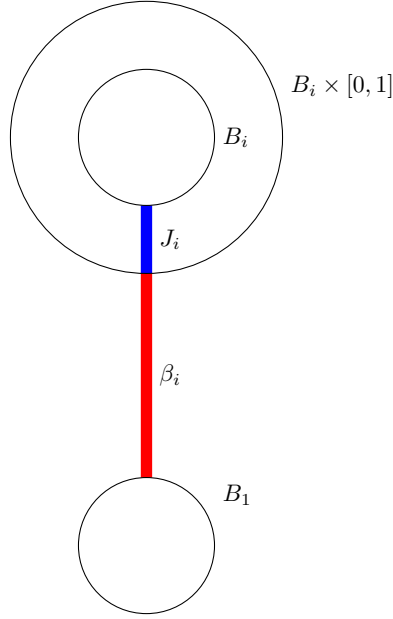
\begin{figure}[h]
\centering

\begin{tikzpicture}[scale=0.45, transform shape]

 \draw [blue, line width=1.5mm] (0,3)--(0,5) (0,4)node[black, anchor=west, scale=2]{$J_i$};
 \draw[red, line width=1.5mm] (0,-3)--(0,3) (0,0)node[black, anchor=west, scale=2]{$\beta_i$};
 \draw (0,7) circle (2cm) (2,7)node[anchor=west, scale=2]{$B_i$};
\draw (0,7) circle (4cm) (4,8.5) node[anchor=west, scale=2]{$B_i\times [0,1]$};
\draw (0,-5) circle (2cm) (2, -3.5) node[anchor=west, scale=2]{$B_1$};

\end{tikzpicture}

\caption{The interval $J_i$ in blue is homotoped so it is included in the axis of rotation of $T_{B_i}$.}

\end{figure}
\FloatBarrier

The same analysis as above yields:

  \[   \tsh_{n,k}([T_{B_i}])[c_j]=\left\{
\begin{array}{ll}
1 & \text{ if } c_j=\beta_i\\
     0 &   \text{ otherwise} . \\
      
\end{array}
\right. \]


Thus, $\tsh_{n,k}$ maps the set $\{[T_{S_1}],...,[T_{S_n}],[T_{B_2}],...,[T_{B_k}]\}$ to a basis of $H^1(M_{{n,k}},\partial M_{{n,k}})$. Recall that $\ts(M_{n,k})$ is an abelian group in which each element has order at most 2 and $\tsh_{n,k}$ restricts to a homomorphism on $\ts(M_{n,k})$.  Therefore, $\{[T_{S_1}],...,[T_{S_n}],[T_{B_2}],...,[T_{B_k}]\}$ is 
linearly independent and we have: $$\langle [T_{S_1}],...,[T_{S_n}],[T_{B_2}],...,[T_{B_k}]\rangle\cong H^1(M_{{n,k}},\partial M_{{n,k}}).$$

Thus, by Proposition \ref{prop Twist generated by TS and TB} we have: $$\ts(M_{{n,k}})\cong H^1(M_{{n,k}},\partial M_{{n,k}}),$$ which is in turn  isomorphic to $\Z_2^{n+k-1}$ by Proposition \ref{prop generating set of H(Mnk)}. This finishes the proof of the proposition.
      
  \end{proof}
\end{prop}

Proposition \ref{prop isomorphism Twist and cohomology} shows that $\tsh_{n,k}$ restricted to $\ts(M_{n,k})$ is an isomorphism. Let $Q:=\tsh_{n,k}\Bigr|_{\ts(M_{n,k})}^{-1}$. 
By construction, the composition $\map(M_{{n,k}})\xrightarrow{\tsh_{n,k}}H^1(M_{{n,k}},\partial M_{{n,k}})\xrightarrow{Q}\ts(M_{{n,k}})$, which we again denote by $\tsh_{n,k}$, restricts to the identity on $\ts(M_{n,k})$. Thus, by  \cite[Lemma 3.1]{3} and Proposition \ref{prop isomorphism Twist and cohomology} we get the following corollary.


\begin{colo}\label{col HV splits}
For $k\ge 1,$ the Hatcher--Vogtmann short exact sequence:
$$ 1\xrightarrow{} \ts(M_{n,k})\xrightarrow{}\map(M_{n,k})\xrightarrow{\rho_{n,k}}\map({G_{n,k}})\xrightarrow{}1$$
  is a split short exact sequence with $\ts(M_{n,k})\cong \Z_2^{n+k-1}$.  
\end{colo}

\section{Computation of the twisting crossed homomorphism on the generating set}\label{sec:Computation twisting crossed}

The following lemma allows us to compute the derivative crossed homomorphism in terms of Cartesian coordinates.

\begin{lema}\label{lemma computaion derivative}Let $N$ and $N^{\prime}$ be smooth 3-manifolds. Let $\sigma_0:N\xrightarrow{}\Fr(TN)$ and $\sigma_0^{\prime}:N^{\prime}\xrightarrow{}\Fr(TN^{\prime})$  be trivializations of $TN$ and $TN^{\prime}$ respectively. Let $f:N\xrightarrow{}N$ and $f^{\prime}:N^{\prime}\xrightarrow{}N^{\prime}$ be diffeomorphisms, and suppose that $f=\phi\circ f^{\prime}\circ\phi^{-1}$, for some diffeomorphism $\phi:N^{\prime}\xrightarrow{}N$ that is compatible up to homotopy with $\sigma_0$ and $\sigma_0^{\prime}$. Then $\mathcal{D}(f)^{-1}= \mathcal{D}(f^{\prime})^{-1}\circ\phi^{-1}.$

    \begin{proof}
Since $\phi$ is compatible with $\sigma_0$ and $\sigma_0^{\prime}$, we have the following commutative diagram:

\[ \begin{tikzcd}[row sep=25,column sep = 40]
N^{\prime} \arrow{r}{\sigma_0^{\prime}} & \Fr(TN^{\prime}) \\
N \arrow{u}{\phi^{-1}} \arrow{r}{\sigma_0}& \Fr(TN)\arrow[swap]{u}{(D(\phi^{-1}))_{*}}
\end{tikzcd}
\]

From this diagram we see that for each $p\in N,$ we have $D_p(\phi^{-1})[\sigma_0(p)]=\sigma_0^{\prime}(\phi^{-1}(p))$, so $\sigma_0(p)=D_{\phi^{-1}(p)}(\phi)\sigma_0^{\prime}(\phi^{-1}(p))$ and $\sigma_0^{-1}(p)=(\sigma_0^{\prime})^{-1}(\phi^{-1}(p))D_p(\phi^{-1}).$  Moreover, we have: $$D_{\phi^{-1}(f(p))}(f^{\prime})^{-1}=D_{\phi^{-1}(f(p))}(\phi^{-1} \circ f^{-1}\circ \phi)=D_{p}(\phi^{-1}) D_{f(p)}f^{-1} D_{\phi^{-1}(f(p))}(\phi).$$ Thus, $$\mathcal{D}(f)^{-1}(p)=\sigma_0^{-1}(p)[Df^{-1}]_{f(p)}\sigma_0(f(p))=$$
$$(\sigma_0^{\prime})^{-1}(\phi^{-1}(p))D_{p}(\phi^{-1})D_{f(p)}f^{-1}D_{\phi^{-1}(f(p))}(\phi)\sigma_0^{\prime}(\phi^{-1}(f(p)))=$$
$$(\sigma_0^{\prime})^{-1}(\phi^{-1}(p))D_{\phi^{-1}(f(p))}(f^{\prime})^{-1}\sigma_0^{\prime}(\phi^{-1}(f(p)))=$$ 

$$(\sigma_0^{\prime})^{-1}(\phi^{-1}(p))D_{f^{\prime}(\phi^{-1}(p))}(f^{\prime})^{-1}\sigma_0^{\prime}(f^{\prime}(\phi^{-1}(p)))=\mathcal{D}(f^{\prime})^{-1}(\phi^{-1}(p)).$$ 

 \end{proof}
\end{lema}

In \cite{4} we proved the following theorems:

\begin{teorema}\label{thm cross F}\cite[Section 3]{4}
    For every $G_{n,k}$ we have that $\tsh_{n,k}([F_{i,j}])=1$.

\end{teorema}

\begin{teorema}\label{thm cross G} \cite[Section 4]{4}
For every $G_{n,k}$ we have that $\tsh_{n,k}([G_j])=1.$
 
\end{teorema}

Let $N$ be the neighborhood associated to $\Phi_{(a_i,J_l)}$ constructed in Section \ref{sec: construction of diffeomorphisms}.
To simplify notation denote $\Phi_{(a_i,J_l)}$ by $\Phi$ and $N^{\prime}$ by $N$.

\begin{teorema}\label{thm cross Phi}

For every $G_{n,k}$ we have that $\tsh_{n,k}([\Phi_{(a_i,J_l)}])=1.$
  \begin{proof}

We study the action of  $\Phi_{(a_i,J_l)}$ on $H_1(M_{{n,k}},\partial M_{{n,k}})$. From the discussion after Proposition \ref{prop generating set of H(Mnk)} a basis for $H_1(M_{{n,k}},\partial M_{{n,k}})$ consists of the images of generators $a_i$ of the fundamental group of $M_{n,k}$ for  $1\le i\le n$, and a choice of $k-1$ independent relative cycles. $\Phi_{(a_i,J_l)}$ is only nontrivial on the relative cycles $b_{l-1}$ and $b_{l}$.
 Let us consider the action of  $\Phi$ on $b_{l-1}$. Notice that the action of  $\Phi$ on $b_l$ is similar.  As in Section \ref{sec: construction of diffeomorphisms}, we can take a solid torus neighborhood $N\subseteq\R^3$ of $B_l$ that looks as follows:

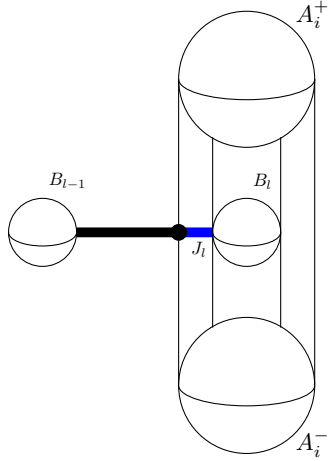
\begin{figure}[h]
\centering

\begin{tikzpicture}[scale=0.45, transform shape]
\draw (-1,-2.2)--(-1,-7.8);
\draw (1,-2.2)--(1,-7.8);
 \draw (-2,-0.5)--(-2,-9.5);
\draw (2,-0.5)--(2,-9.5);
\draw[ blue, line width=1.3mm] (-1,-5)--(-2,-5)
(-1.4,-6)node[black, anchor=south, scale=1.5]{$J_l$};
  \draw (0,-0.5) circle (2cm)(1.2,1.4)node[anchor=west, scale=2]{$A_i^{+}$};
  \draw (-2,-0.5) arc (180:360:2 and 0.6);
  \draw (0,-9.5) circle (2cm)(1.2,-11.2)node[anchor=west,scale=2]{$A_i^{-}$};
  \draw (-2,-9.5) arc (180:360:2 and 0.6);
  \draw (0,-5) circle (1cm)(0,-3.5)node[anchor=west, scale=1.5]{$B_l$};
  \draw (-1,-5) arc (180:360:1 and 0.4);
  \draw (-6,-5) circle (1cm)(-6,-3.5)node[anchor=west, scale=1.5]{$B_{l-1}$};
  \draw (-7,-5) arc (180:360:1 and 0.4);

  \draw[ line width=1.3mm] (-2,-5)--(-5,-5);
  
  \fill[fill=black] (-2,-5) circle (7pt);
\end{tikzpicture}

\caption{The interval $J_l$ is shown in blue.}

\end{figure}

\FloatBarrier



Since $\Phi_{(a_i,J_l)}$ has the same local formula as $F_{i,j}$, the computations from Theorem \ref{thm cross F} also show that  $\tsh_{n,k}([\Phi_{(a_i,J_l)}])([b_{l-1}])=0$. Similarly, $\tsh_{n,k}([\Phi_{(a_i,J_l)}])([b_{l}])=0$. Since the action of  $\Phi_{(a_i,J_l)}$   is only nontrivial on the relative cycles $b_{l-1}$ and $b_{l}$,  $\tsh_{n,k}([\Phi_{(a_i,J_l)}]) (\gamma)=0$   for all $\gamma\in H_1(M_{{n,k}},\partial M_{{n,k}})$. Thus, $\tsh_{n,k}([\Phi_{(a_i,J_l)}])$ is trivial.
      
  \end{proof}  
\end{teorema}

We now turn to $M_{\Gamma}$. Let $N$ be the neighborhood associated to $H$ constructed in Section \ref{sec: construction of diffeomorphisms}.
By Lemma \ref{lemma computaion derivative} we can compute the derivative crossed homomorphism of $H$ by computing the derivative crossed homomorphism of $h^{\prime}:N^{\prime}\xrightarrow{}N^{\prime}$ defined by:

$$h^{\prime}(x,y,z)=(x+4\psi(\sqrt{y^2+z^2}),y,z).$$

To simplify notation denote $h^{\prime}$ by $H$ and $N^{\prime}$ by $N$ in the following theorem.

\begin{teorema}
\label{thm crossed H}
      
If $|E_{\ell}(\Gamma)|\ge 2$, then we have $\tsh([H])=1$.
      
\begin{proof}

Let $i \in \Z$ and consider the loop $a_i$. Choose  $\gamma\in[a_i]$  such that ${\gamma}$ intersects $N$ as $(\tau,0, 0)$, for $\tau\in[-\epsilon,\epsilon]$. Recall that for $(x,y,z)$ in $N$ we have that $H^{-1}$ has  the formula: $$H^{-1}(x,y,z)=(x-4\psi(\sqrt{{z}^{2}+y^{2}}),y,z).$$ 

Then we obtain:
$${\mathcal{D}([H])^{-1}}(\gamma(t))=
\begin{bmatrix}

1&-\dfrac{4y}{\sqrt{y^{2}+z^{2}}}\dfrac{d\psi}{d r}(\sqrt{y^{2}+z^{2}}) & -\dfrac{4z}{\sqrt{y^{2}+z^{2}}}\dfrac{d\psi}{dr}(\sqrt{y^2+z^2})\\
0 & 1 & 0\\
0 & 0 &1 \\
\end{bmatrix}
(\tau,0, 0)=$$

$$\begin{bmatrix}

1&0 & 0\\
0 & 1 & 0\\
0 & 0 &1 \\
\end{bmatrix}(\tau,0, 0)= \begin{bmatrix}

1&0 & 0\\
0 & 1 & 0\\
0 & 0 &1 \\
\end{bmatrix}.
$$
 Therefore, $[\mathcal{D}([H])(\gamma)]$ is trivial for every $\gamma$ in $N^{\prime}$. Since $H$ is the identity outside $N^{\prime}$, $\mathfrak{T}([H])[\gamma]$ is trivial in $\pi_{1}(\GLR,\mathrm{Id})$ for every $\gamma\in\pi_1(M_{\Gamma})$.
Therefore, the twisting crossed homomorphism $\mathfrak{T}$ evaluated  at $H$, $\mathfrak{T}([H])$, is trivial.

    \end{proof}

\end{teorema}

Let $N$ be the neighborhood associated to $P_{i,j}$ constructed in Section \ref{sec: construction of diffeomorphisms}.
By Lemma \ref{lemma computaion derivative} we can compute the derivative crossed homomorphism of $P_{i,j}$ by computing the derivative crossed homomorphism of $p_{i,j}^{\prime}:N^{\prime}\xrightarrow{}N^{\prime}$  defined by: 

\[   p_{i,j}^{\prime}(p)=\left\{
\begin{array}{ll} 
   \nu(p) & \text{if } p\in N_i^{\prime}\\
   \nu^{-1}(p) & \text{if } p\in N_j^{\prime}\\
       (\theta+\psi(r)\pi, \varphi, r)& \text{if } p\in N^{\prime} -(N_i^{\prime}\cup N_j^{\prime}) , \\

\end{array} 
\right. \] where $p\in N^{\prime} -(N_i^{\prime}\cup N_j^{\prime})$ has spherical coordinates $(\theta, \varphi,r)$.
To simplify notation denote $p_{i,j}^{\prime}$ by $P$ and $N^{\prime}$ by $N$, $N_i^{\prime}$ by $N_i$, and $N_j^{\prime}$ by $N_j$ in the following theorem.

\begin{teorema}
\label{thm cross Pij}
    If $|E(\Gamma)|<\infty$, then we have $\tsh([P_{i,j}])=1.$

    \begin{proof}
          For $\gamma(t)$ a loop in $N_i^{\prime}$, $P(\gamma(t))$ is the loop $\nu(\gamma(t))$. Since   $P^{-1}=\nu^{-1}$ on $N_i$, and $p \mapsto D_p\nu$ is homotopic to a constant map, we have that $[\mathcal{D}(P^{-1})(\gamma)]$ is trivial for every $\gamma$ in $N_i$. If $\gamma\in N_j$, the same argument shows that $[\mathcal{D}(P^{-1})(\gamma)]$ is trivial. Since $P$ is the identity outside of $N^{\prime}$,  $\mathfrak{T}([P])[\gamma]$ is trivial in $\pi_{1}(\GLR, \mathrm{Id})$ for every $\gamma\in \pi_1(M_{\Gamma})$. Therefore, the twisting crossed homomorphism $\mathfrak{T}$ evaluated  at $P$, $\mathfrak{T}([P])$, is trivial.
    \end{proof}
\end{teorema}

\section{Computation of the group sections $s_{n,k}$ and $s$}\label{sec: section s and snk}

In this section we will prove Theorems \ref{Hatcher Vogtmann split}--\ref{thm section lift of pij}.
\begin{proof}[Proof of Theorem \ref{Hatcher Vogtmann split}]

Since the Hatcher--Vogtmann short exact sequence splits by Corollary \ref{col HV splits}, there exists a section $$s_{n,k}:\map(G_{n,k})\xrightarrow{}\map(M_{n,k}).$$ 
Since the maps $[F_{i,j}], [G_j],$ and $[\Phi_{(a_i,J_l)}]$ lift the maps $[R_{i,j}], [I_{j}]$, and $ [\varphi_{(a_i,J_l)}]$ respectively,   the map $s_{n,k}$ is given by:
$$s_{n,k}([R_{i,j}])=\tsh_{n,k}([F_{i,j}]^{-1})[F_{i,j}],s_{n,k}([I_j])= \tsh_{n,k}([G_j]^{-1})[G_j],$$ \text{ and }  $$s_{n,k}([\varphi_{(a_i,J_l)}])= \tsh_{n,k}([\Phi_{(a_i,J_l)}]^{-1})[\Phi_{(a_i,J_l)}].$$ 
 By Theorems \ref{thm cross F}--\ref{thm cross Phi} we have that $\tsh_{n,k}( [F_{i,j}])=1, \tsh_{n,k}([G_j])=1$ and $\tsh_{n,k}([\Phi_{(a_i,J_l)}])=1.$ Then, $s_{n,k}([R_{i,j}])=[F_{i,j}]$, $s_{n,k}([I_j])=[G_j]$, and $s_{n,k}([\varphi_{(a_i,J_l)}])=[\Phi_{(a_i,J_l)}].$ 
\end{proof}

 Let $\Gamma^{\prime}$ be a thorned subgraph of $\Gamma,$ and
  consider the map $\iota:\map(\Gamma^{\prime})\xrightarrow{}\map_c(\Gamma)$ defined as follows. 
 For $[\phi]\in \map(\Gamma^{\prime})$, $\iota([\phi])$ is defined to be the equivalence class of the map:

  \[   \overline{\phi}(p):=\left\{
\begin{array}{ll}
\phi(p) & \text{if } p\in \Gamma^{\prime}\\
     p &  \text{otherwise}. \\
      
\end{array}
\right. \] Since the homotopy equivalences in $\map(\Gamma^{\prime})$ fix the marked points of $\Gamma^{\prime}$, this map is well defined. 

 Let $M_{n,k}$ be included in $M_\Gamma,$ and consider the map $\theta:\map(M_{n,k})\xrightarrow{}\map_c(M_{\Gamma})$ defined as follows. For $[f]\in \map(M_{n,k})$, $\theta([f])$ is defined to be the equivalence class of the map:

\[   \overline{f}(p):=\left\{
\begin{array}{ll}
 f(p) & \text{if } p\in M_{n,k} \\
     p &   \text{otherwise}. \\
      
\end{array}
\right. \]  Since the diffeomorphisms  in $\map(M_{n,k})$ fix the boundary of $M_{n,k}$, this map is well defined. Notice that $\theta$ is not injective. To see this, we observe that if a boundary sphere $S$ of $M_{n,k}$ separates $M_\Gamma$, then $\theta([T_S])$ is trivial in $\map(M_\Gamma)$. Notice that if $[\phi]\in \map_c(\Gamma)$ and $\Gamma^{\prime}$ is a thorned subgraph containing $\supp (\phi)$, then $\iota([{\phi\Bigr|}_{\Gamma^{\prime}}])=[\phi].$ Similarly, if $[f]\in \map(M_{\Gamma})$ has compact support and $\supp(f)$ is contained in  $M_{n,k}$,  then $\theta([{f\Bigr|}_{M_{n,k}}])=[f].$
We have the following proposition:

\begin{prop}\label{prop lift s mapc}
Let $[\phi]\in \map_c(\Gamma)$ with $\Gamma^{\prime}$ a thorned subgraph containing $\supp(\phi)$ and a homotopy equivalence $\alpha:\Gamma^{\prime}\xrightarrow{} G_{n,k}$. Then $s^{\prime}([\phi]):=\theta\circ s_{n,k}\circ \alpha_{*}([\phi\Bigr|_{\Gamma^{\prime}}])$ is a lift of $[\phi].$ 
\end{prop}

\begin{proof}
 Let $\rho_{n,k}^{\prime}=\alpha_{*}^{-1}\circ \rho_{n,k}$. Consider the following commutative diagram:

\[\begin{tikzcd}
\map(\Gamma^{\prime})\arrow{r}{\iota} & \map(\Gamma) \\
\map(M_{n,k})\arrow{u}{\rho_{n,k}^{\prime}}\arrow{r}{\theta}& \map(M_\Gamma)\arrow{u}{\Psi}.\\
\end{tikzcd}\]
 Thus, $\Psi(s^{\prime}([\phi]))=\Psi(\theta(s_{n,k}\circ \alpha_{*}([\phi\Bigr|_{\Gamma^{\prime}}])))=\iota(\alpha_*^{-1}\rho_{n,k}(s_{n,k}\circ \alpha_{*}([\phi\Bigr|_{\Gamma^{\prime}}])))=\iota([\phi\Bigr|_{\Gamma^{\prime}}])=[\phi].$
 It follows that  $s^{\prime}([\phi])$ is a lift of $[\phi].$ 
 
 \end{proof}
 

Let $[\phi]\in \map_c(\Gamma)$ with $\Gamma^{\prime}$ a thorned subgraph containing $\supp(\phi)$  and a homotopy equivalence $\alpha:\Gamma^{\prime}\xrightarrow{}G_{n,k}$. By Proposition \ref{prop lift s mapc}, $s^{\prime}([\phi])$ is a lift of $[\phi]$, so $s([\phi])=\tsh(s^{\prime}([\phi])^{-1})s^{\prime}([\phi])$. We have the following proposition.



\begin{prop}\label{prop s for compactly supp} Let $[\phi]\in \map_c(\Gamma)$ with $\Gamma^{\prime}$ a thorned subgraph containing $\supp(\phi)$ and a homotopy equivalence $\alpha:\Gamma^{\prime}\xrightarrow{} G_{n,k}$. Then: 
$$s([\phi])=s^{\prime}([\phi]).$$
   \begin{proof}

   We have the following commutative diagram:

\[\begin{tikzcd}
\ts(M_{n,k})\arrow{r}{\theta} & \ts(M_{\Gamma})\\
\map(M_{n,k})\arrow{u}{\tsh_{n,k}}\arrow{r}{\theta}& \map(\mg)\arrow{u}{\tsh}.\\
\end{tikzcd}\] 

 It follows from this diagram that for a map $f\in\map(M_{n,k})$ we have  that if $\theta(\tsh_{n,k}([f]))=1$, then  $\tsh(\theta([f]))=1$. Since $\tsh_{n,k}(s_{n,k}\circ \alpha_{*}([\phi\Bigr|_{\Gamma^{\prime}}]))=1$ by  Theorems \ref{thm cross F}--\ref{thm cross Phi}, $\tsh(\theta(s_{n,k}\circ \alpha_{*}([\phi\Bigr|_{\Gamma^{\prime}}])))=\tsh(s^{\prime}([\phi]))=1.$
Thus, $s([\phi])=s^{\prime}([\phi]).$

\end{proof} 
   
\end{prop}


\begin{proof}[Proof of Theorem \ref{thm limit si}]
   Since $\phi_i$ has compact support, we deduce from Proposition \ref{prop s for compactly supp} that  $s([\phi_i])=s_i([\phi_i])$. By \cite[Lemma 5.7]{7}, $\tsh$ is continuous. Since convergent sequences in $\map(\Gamma)$ lift to convergent sequences in $\map(\mg)$ by the topological splitting in [12], $s$ is continuous, and  then we have:

 $$s\left([\lim\limits_{i \to \infty} \phi_i]\right)=\lim\limits_{i \to \infty}s([\phi_i])=\lim\limits_{i \to \infty}s_i([\phi_i]).$$

This finishes the proof of the theorem.

\end{proof}

\begin{proof}[Proof of Theorems \ref{thm section lift of h} and \ref{thm section lift of pij}]
    
We have that on loop shifts and end swaps of ends of the same type $s$ is given by $s([h])=\mathfrak{T}([H]^{-1})[H]$, and $s([p_{i,j}])=\mathfrak{T}([P_{i,j}]^{-1})[P_{i,j}]$ respectively.  
By  Theorems \ref{thm crossed H} and  \ref{thm cross Pij} we have that $\tsh([H])=1\text{ and }\tsh([P_{i,j}])=1$. Then $s([h])=[H]$ and $s([p_{i,j}])=[P_{i,j}]$.  This finishes the proof.

\end{proof}

\bibliographystyle{abbrv}
\bibliography{References}
\end{document}